\documentclass[11pt]{article}
\usepackage{KUL_style}
\setcitestyle{authoryear,round}
\usepackage{lipsum}  
\usepackage{datetime}
\usepackage{graphicx}
\usepackage{amsmath} 
\usepackage{amssymb} 
\usepackage{amsthm}  
\newtheorem{proposition}{Proposition}
\usepackage{bm}      
\usepackage{optidef} 
\usepackage{multirow} 
\usepackage{caption}  
\usepackage{pgfplots} 
\usepackage{tikz} 
\usetikzlibrary{calc,spy} 
\usepackage{lscape} 
\usepackage{tabularx} 
\usepackage{adjustbox} 
\usepackage{xcolor} 
\usepackage{natbib} 
\usepackage{algorithm} 
\usepackage{algpseudocode} 
\usepackage{tablefootnote} 
\pgfplotsset{compat=1.18} 
\allowdisplaybreaks[4]
\usepackage{booktabs} 
\usepackage{longtable} 
\usepackage{array} 

\title{Joint Planning and Scheduling of Modular Vehicles for Passenger--Freight Integration}

\author{
Wanru Chen$^{1}$, Jiaming Wu$^{1}$, Balázs Kulcsár$^{2}$\\[0.4em]
\begin{minipage}{0.95\textwidth}
\centering
\small
$^{1}$Architecture and Civil Engineering, Chalmers University of Technology, Gothenburg, Sweden\\
$^{2}$Electrical Engineering, Chalmers University of Technology, Gothenburg, Sweden
\end{minipage}
}

\date{}

\date{\vspace{-4.0em}}

\begin{document}

\maketitle
\pagestyle{plain}
\pagenumbering{arabic}
\selectlanguage{american}
\begin{abstract}
This paper proposes a modular vehicle system for passenger--freight integration along a bidirectional transit corridor. 
The system uses homogeneous units that can be coupled into vehicles and assigned to either passenger or freight service. 
Freight is carried by dedicated units, with loading and unloading coordinated with docking and undocking and separated from passenger boarding and alighting. 
To better respond to uncertain passenger demand and integrate freight transport, vehicles can be reconfigured at intermediate stations, where they can also depart and terminate.
We jointly optimize departure-specific service routes, timetables, vehicle compositions, unit schedules, and passenger--freight demand assignments, with unit reuse constrained by explicit docking and undocking times. 
These decisions are modeled on a space--time--state network and formulated as a stochastic mixed-integer program that minimizes unit deployment costs, passenger waiting costs, and penalties for unmet freight demand.
Passenger demand uncertainty is addressed using linearized chance constraints. 
To solve the problem, we develop an exact Benders decomposition algorithm with valid inequalities and a warm-start strategy, together with a tailored decomposition-based heuristic for larger instances.
Computational experiments on instances generated from representative transit corridors in Gothenburg demonstrate the effectiveness of the Benders algorithm for small- and medium-sized instances and the scalability of the heuristic for larger problems.
Sensitivity analyses highlight the value of accounting for passenger demand uncertainty and the effects of temporal overlap between passenger and freight demand. 
Comparisons with benchmark transit systems further demonstrate the operational advantages of the proposed modular integrated system.
\end{abstract}

\selectlanguage{american} 


\section{Introduction}

Urban public transport services exhibit pronounced spatiotemporal fluctuations in passenger demand, creating residual capacity during off-peak periods and on less-utilized corridor segments. Passenger--freight integration has therefore attracted increasing attention as a means to improve system-wide resource utilization by using residual passenger transport capacity for urban freight \citep{mo2024synergising}. 
Yet realizing such co-transport in urban transit systems is challenging with conventional fixed-capacity vehicles, which are primarily designed and scheduled for passenger service. 
A fixed allocation of onboard space may conflict with passenger demand, while freight loading and unloading can increase dwell times and affect service reliability. 
These limitations motivate the use of modular vehicles, which provide a more flexible capacity structure for coordinating passenger and freight services \citep{zheng2025last}.

A modular vehicle system uses standardized units to form vehicles with flexible capacities. 
Standardized units can be coupled or decoupled at stations, allowing a vehicle formation to expand on high-demand passenger segments and shorten when demand is low. Coupled passenger units are connected by interior doors, allowing passengers to move freely within the formed vehicle \citep{wu2021modular, li2022trajectory}. 
This structure can be extended to passenger--freight integration by assigning homogeneous units to different functions, as shown in Figure~\ref{fig:modular adaption example}, so that residual passenger transport capacity can be converted into freight-carrying capacity. In this proposed system, freight is carried in standardized boxes or parcels within freight-assigned units, which helps separate passenger and freight operations and reduce operational conflicts \citep{lin2024modular}. 
The function of each unit can be adjusted between passenger and freight service over time through operational assignment and vehicle composition decisions, thereby improving system utilization and reducing the need for separate freight delivery operations \citep{gao2026evaluating}.

\begin{figure}[t]
\centering
\includegraphics[width=0.95\linewidth]{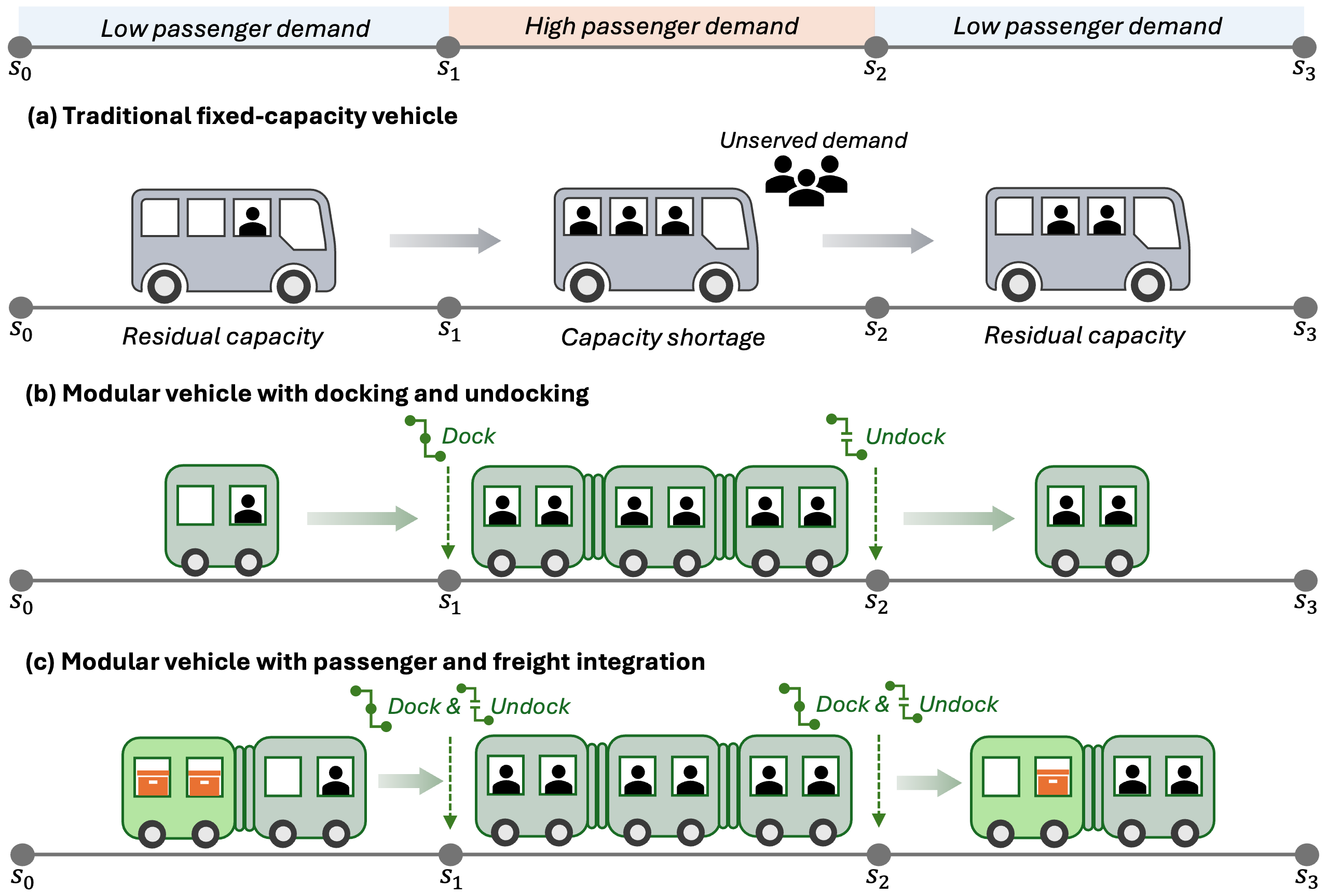}
\caption{Illustration of capacity adaptation via docking and undocking in modular vehicles.}
\label{fig:modular adaption example}
\end{figure}

Planning and scheduling are central to realizing the operational value of modular passenger--freight systems, as they determine how service quality and operating costs are balanced. 
In modular vehicle operations, these decisions jointly involve flexible service routes, station-level timetables, segment-level vehicle compositions, and unit schedules. 
A service may start or terminate at an intermediate station, its arrival and departure times must be determined along the served segments, and its composition must specify both the number of coupled units and their allocation between passenger and freight functions. 
At the same time, these planned services are feasible only if the required units can be stored, reused, docked, and undocked at the right stations and times. 
This creates a tightly coupled planning and scheduling problem that cannot be captured by models treating routing, timetabling, vehicle composition, and unit circulation separately. 
Existing modular vehicle studies have addressed parts of this structure, but the joint treatment of these interdependent decisions remains limited.

These planning and scheduling decisions are ultimately driven by the demand they serve. 
Passenger demand is usually uncertain and sensitive to service quality, while freight demand can often be served more flexibly by using spare capacity \citep{cacchiani2020robust,li2026joint}. A key practical challenge is to handle freight without disrupting passenger service. Many studies simplify loading and unloading by assuming item-by-item handling within normal dwell times \citep{lin2024modular, li2024scheduling}. In a modular system, freight can be carried in dedicated units, so loading and unloading can be managed through unit docking and undocking at stations rather than within passenger dwell times.

Figure~\ref{fig:modular integration illustration} illustrates coordinated passenger--freight operations in the proposed modular system. At $S_1$, the vehicle arrives from a low-demand segment with one passenger unit and one freight unit. Before the next high-demand segment, the freight unit is undocked and unloaded, while ready units from the station pool are docked as passenger units to increase capacity. Since passengers can move between coupled units, they can redistribute based on destinations before reaching $S_2$. At $S_2$, units of passengers who arrived at their desired stop are undocked without affecting passengers who continue their trips, and a pre-loaded freight unit is docked for the following low-demand segment. 
The figure highlights three modeling challenges: Passenger service is provided by the coupled vehicle as a whole, freight handling relies on dedicated units, and capacity changes require docking and undocking operations with explicit process times and station-level unit availability.

\begin{figure}[t]
\centering
\includegraphics[width=0.95\linewidth]{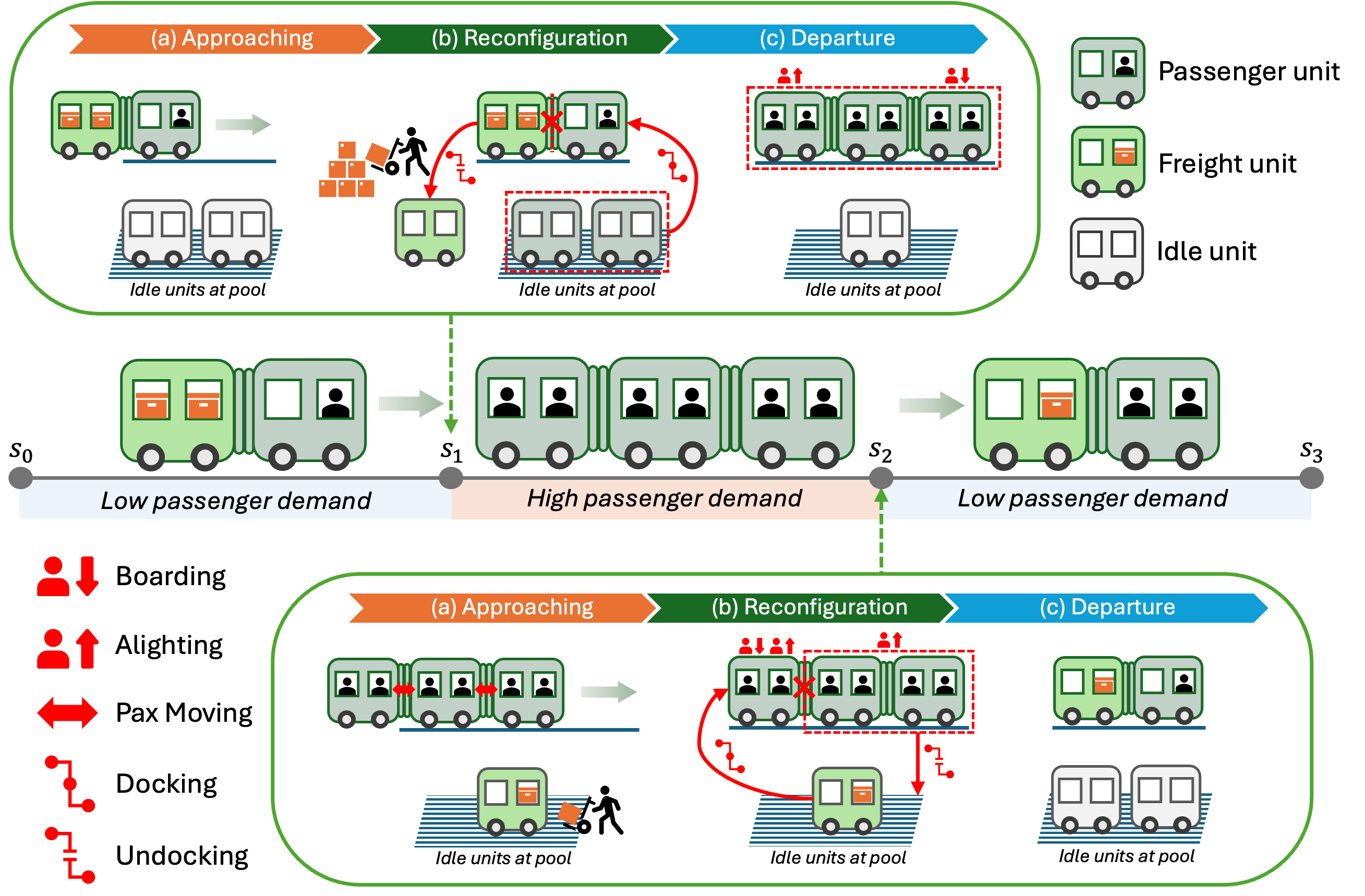}
\caption{Illustration of passenger--freight operations and vehicle reconfiguration in a modular vehicle system.}
\label{fig:modular integration illustration}
\end{figure}

Building on these system features, this study addresses the joint planning and scheduling of modular vehicles for passenger--freight integration on a bidirectional transit corridor. Each station maintains a pool of modular units that can be deployed by services in either direction. The problem jointly determines flexible departure-specific service routes, station-level timetables, segment-level vehicle compositions, unit schedules, and passenger--freight demand assignments. Passenger demand is stochastic, whereas freight demand is deterministic and may be only partially served. Freight loading, unloading, and transfer are executed via the docking and undocking of dedicated freight units.

The methodological challenge is that reconfiguration decisions are neither purely service-design decisions nor purely vehicle-scheduling decisions. A departure may exist only if passenger and freight capacities, station-level unit availability, docking and undocking times, and downstream reuse opportunities are simultaneously consistent. 
To model these interdependent decisions, we propose a tailored space--time--state network and a stochastic mixed-integer network-flow formulation that supports different operational settings through adaptations of the network structure and model constraints. 
Passenger assignment is reformulated in a path-based flows and enforced through chance constraints. 

The proposed network reveals two key structures that guide the solution approach. First, service and reconfiguration planning can be distinguished from station-wise unit circulation, which motivates an exact multi-cut Benders decomposition algorithm strengthened with valid inequalities and a warm-start strategy. Second, passenger and freight services follow different network structures but are unified through the global circulation of modular units, which motivates a two-stage decomposition-based heuristic for large instances under a passenger-priority setting. 
We evaluate the proposed model and algorithms on instances motivated by representative transit corridors in Gothenburg. The experiments assess algorithmic performance, quantify the impacts of passenger-demand uncertainty and passenger--freight temporal overlap, and demonstrate the operational benefits of modular passenger--freight integration relative to benchmark systems. 

The contributions of this paper are threefold:
\begin{itemize}
    \item We develop a joint planning and scheduling framework for modular vehicles that adapts service capacity to heterogeneous passenger and freight demand while ensuring feasible unit operations.
    \item We address the complex coupling among decisions through decomposition-based solution methods, including an exact Benders decomposition algorithm enhanced by valid inequalities and a  two-stage heuristic for large instances.
    \item We evaluate the model and algorithms on realistic corridor instances, deriving operational insights and demonstrating the value of modular vehicles for urban passenger--freight integration.
\end{itemize}

The remainder of this paper is organized as follows. Section~\ref{Literature Review} reviews the related literature, and Section~\ref{sec: Problem Statement} introduces the problem setting. Sections~\ref{Network Construction} and~\ref{Model Formulation} present the network construction and model formulation, followed by the solution methodology in Section~\ref{Solving Methodology}. Sections~\ref{sec:exper design} and~\ref{Computational Performance} report the experimental design and computational results, including sensitivity analysis in Section~\ref{Sensitivity Analysis} and comparisons with benchmark systems in Section~\ref{Comparison with Systems}. Section~\ref{Conclusions and Future Research} concludes and discusses future research directions.

\section{Literature Review}
\label{Literature Review}

This section reviews the literature related to modular passenger--freight integration. 
Urban passenger--freight studies establish the application context and reveal the limits of fixed service structures, while modular vehicle planning and scheduling studies provide the operational basis for flexible capacity reconfiguration. 
Together, these two streams motivate the integrated modeling and solution approach developed in this paper.

\subsection{Passenger and Freight Integration in Urban Transit Systems}

Integrating passenger and freight flows in urban public transit has attracted increasing attention as a way to make better use of existing capacity and to reduce the cost and environmental impacts of operating separate urban freight services. The basic idea is to use residual transit capacity to serve freight demand while maintaining acceptable passenger service quality. Existing studies examine this concept in bus systems, metro systems, and, more recently, modular vehicle systems.

In bus systems, integration is often implemented by using scheduled buses to carry parcels through existing services or services with minor adjustments. Studies examine service adaptation, such as selecting routes or trips for parcel transport \citep{zeng2022optimization, machado2023integration, machado2025operational}, and capacity sharing decisions that coordinate passenger demand, freight demand, and onboard space allocation \citep{li2022capacity}. Recent work also considers uncertainty through stochastic frameworks with passenger priority and time varying capacity constraints \citep{lee2025integrated, zhang2026two}.

In metro systems, integration is mainly studied through shared rail capacity and underground logistics operations, with potential benefits for reducing road based freight emissions \citep{di2023research}. Studies consider freight movements through subway networks under constraints related to train timing, capacity, transfers, safety, and station operations \citep{zhou2025using}, shared passenger and freight transport during off peak hours with flexible train composition \citep{li2024scheduling}, and separate freight train services with limited adjustments to passenger timetables \citep{li2026joint}.

Overall, studies on bus systems and metro systems provide important foundations for integration in public transit. However, flexibility is constrained by the physical and operational structure of each mode. In bus systems, freight is typically inserted into existing vehicles or selected trips. In metro systems, operations must follow fixed rail infrastructure, strict timetables, and limited station handling opportunities. As a result, most models treat passenger and freight integration as allocating residual capacity within a largely fixed service structure, rather than reconfiguring capacity dynamically across service segments.

Modular vehicle systems offer a promising way to overcome the capacity inflexibility of conventional bus- and metro-based integration, since vehicle capacity can be adjusted through docking and undocking.
In a transit corridor setting, \citet{lin2024modular} study passenger--freight co-transportation with modular vehicles and optimize dispatching, module allocation, and station-level docking and undocking decisions.
Their study demonstrates the potential of modular reconfiguration, but reconfiguration and freight-handling times are not fully incorporated into scheduling feasibility.
As a result, the effects on unit availability, vehicle departures, and subsequent vehicle formations may be underestimated. 
Related studies also examine modular passenger--freight services in routing contexts, including modular vehicle routing for combined passenger and freight transport \citep{hatzenbuhler2023modular} and extensions to first-mile and last-mile ridesharing and distribution \citep{sun2025integrating,zheng2025last}. 

\subsection{Planning and Scheduling of Modular Vehicles}

In public transport systems, planning and scheduling are closely related but refer to different decisions. 
Planning determines what services are provided and how capacity is supplied, whereas scheduling determines how operational resources are coordinated to deliver these services. 
In modular vehicle systems, planning includes decisions on service coverage, timetables, and vehicle formation, while scheduling determines how individual units are deployed and circulated over time and space.
These decisions are more tightly coupled than in conventional transit systems because any planned capacity change must be supported by feasible unit movements and station-level operations.

Given the computational difficulty of joint planning and scheduling, much of the modular-vehicle literature has focused on planning-level capacity adaptation. 
\citet{chen2020operational} and \citet{shi2021operations} jointly optimize headways and vehicle capacities under oversaturated demand, showing that modularity can better match supply to time-varying passenger demand.
\citet{gong2021transfer} extends this direction by incorporating passenger route assignment in a transfer-based customized bus system. 
Other studies move from departure-level capacity adjustment to en-route reconfiguration.
\citet{chen2021designing} and \citet{chen2022continuous} model corridor operations with station-wise docking and undocking, while \citet{tian2022planning} studies multi-period decisions on the locations and capacities of stations supporting modular operations.
\citet{khan2025stopless} proposes a stop-less autonomous modular bus service, where modular units attach and detach at stops to reduce unnecessary stopping.
These studies demonstrate the planning value of modular vehicles, but mainly focus on passenger-oriented service design and capacity allocation, with limited consideration of unit circulation and reuse over space and time.

When unit availability and movements are modeled explicitly, modular vehicle planning becomes tightly linked with scheduling. 
Formation changes require docking and undocking operations, and service plans are feasible only if the required units can be circulated, rebalanced, stored, and reused at the right stations and times. 
\citet{tian2023joint} jointly optimize service schedules and vehicle formation on a single line while considering limited unit availability and rebalancing. 
\citet{liu2021improving} extends modular operations to flex-route services, and \citet{tang2024optimisation} studies a hybrid system in which modular autonomous vehicles decouple from and recouple to a fleet traveling along a fixed base route. 
\citet{xia2026integrated} address integrated timetabling and scheduling under uncertainty in a modularized bus network, allowing units to be docked, undocked, and rerouted across lines. 
These studies move toward integrated planning and scheduling, but most are developed for passenger-only services, fixed-line settings, or predefined service structures.

Beyond transit service design, modularity has also been studied in broader routing and resource-coordination problems, including last-mile delivery \citep{rezgui2019application}, emergency medical services \citep{hannoun2022modular}, modular vehicle routing with capacity and time-window constraints \citep{zhou2025modular}, and hub-and-spoke public transport with modular autonomous vehicles \citep{wang2025optimizing}. 

The modular passenger--freight integration considered in this paper further expands the decision scope. 
Existing modular transit models often rely on terminal-to-terminal operations, where timetabling and unit circulation can be organized around terminal departures, arrivals, and end-of-line reuse. 
However, modular services may start or terminate at intermediate stations to better match spatially uneven passenger and freight demand. 
This changes the problem structure: timetabling must determine arrivals and departures at intermediate stations, while unit scheduling must track unit availability, storage, reuse, docking, and undocking beyond terminals. 
The coupling is further strengthened when formations specify passenger and freight unit allocations on each segment, and when docking and undocking process times affect departure feasibility and downstream unit availability. 
Thus, modular passenger--freight integration requires the joint determination of departure-specific service routes, station-level timetables, segment-level passenger/freight formations, and feasible unit schedules. 

From a methodological perspective, integrated modular vehicle planning and scheduling problems are typically formulated as large-scale mixed-integer models, which require problem-specific solution strategies.
Exact formulations with general-purpose solvers can handle small instances, as shown by the MILP reformulation of \citet{tian2023joint} and the MINLP solution procedure of \citet{tang2024optimisation}.
For realistic instances, studies often introduce decomposition, rolling-horizon optimization, and tailored heuristics to reduce computational burden \citep{liu2021improving,tian2023joint,xia2026integrated}.
These studies indicate that no generic algorithm can efficiently solve this class of tightly coupled problems. Solution methods must instead be tailored to the structure of the model.
In this paper, the coupling among decisions motivates decomposition-based methods that yield exact solutions for small instances and efficient solutions for larger instances.

\section{Problem Statement}
\label{sec: Problem Statement}

This study considers the joint planning and scheduling of modular vehicles for integrated passenger and freight services on a bidirectional corridor. 
The problem decides service coverage over corridor segments, station arrival and departure times, passenger and freight unit combinations on each segment, and the storage, reuse, docking, and undocking of individual units at stations.
These decisions are strongly coupled because station-level unit availability determines whether a planned vehicle formation can be operated, while docking and undocking process times affect both vehicle departures and unit reuse. 
The coupling is further shaped by the different operating requirements of passengers and freight, with passengers using shared capacity across coupled units, while freight relying on dedicated units for loading, unloading, and transfer at stations.

This section first introduces the key elements of the problem setting, including the transit corridor, the modular vehicle system, and passenger and freight demand characteristics. These are necessary for understanding the problem and clarifying the  scope of the optimization problem studied in this paper.

\subsection{Transit Corridor}

We consider a bidirectional transit corridor with a set of stations \(S\), indexed by \(s, s' \in S\). The two end stations are terminals, and the remaining stations are intermediate stops arranged sequentially between the terminals.

\begin{figure}[ht]
\centering
\includegraphics[width=0.9\linewidth]{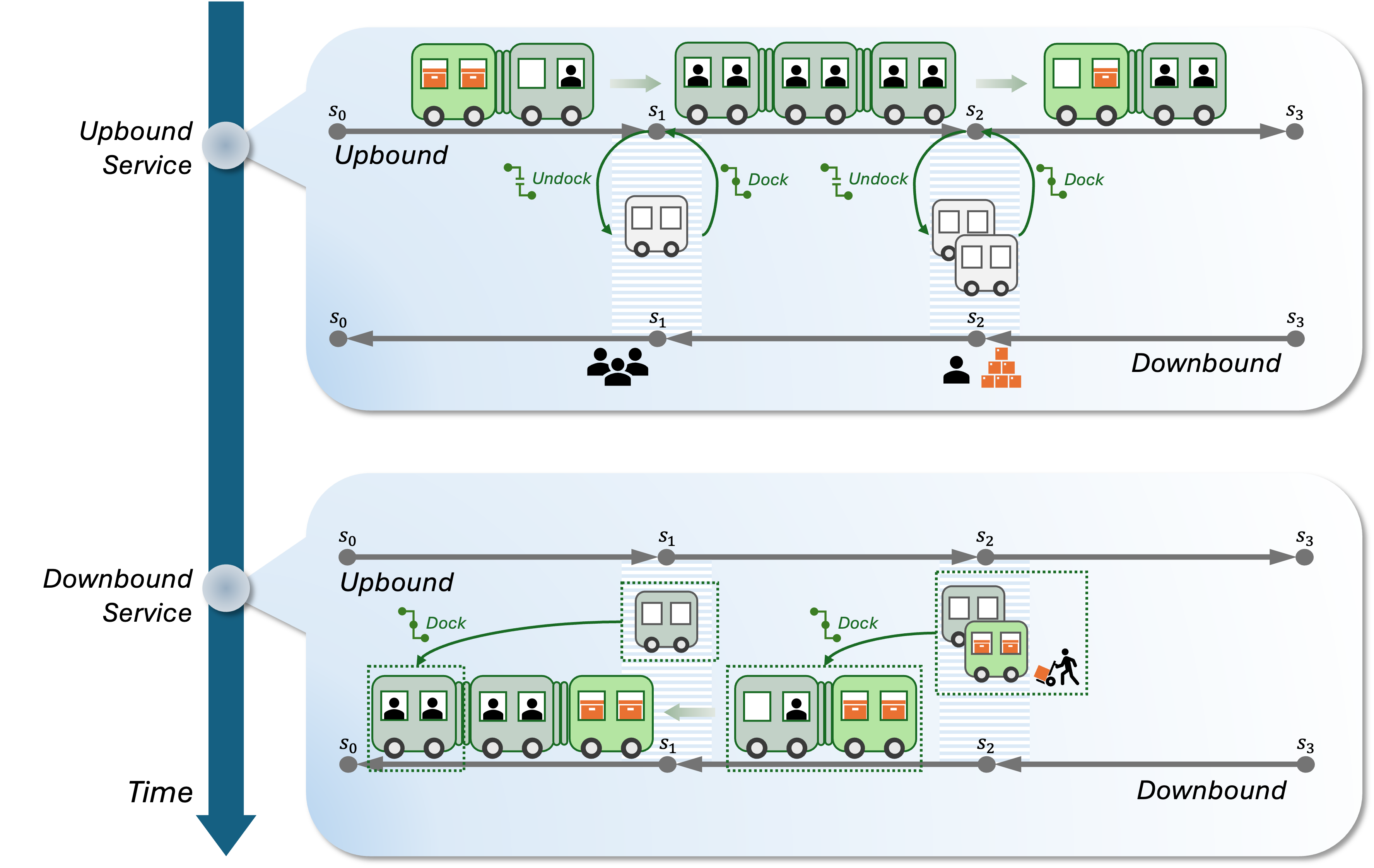}
\caption{Illustration of a modular vehicle system on a bidirectional transit corridor.}
\label{fig:modular bidirectional illustration}
\end{figure}

Each station has two dedicated platforms or tracks, one for upbound vehicles and the other for downbound vehicles. 
Each station \(s \in S\) also has a limited storage capacity \(\overline{n}_s\) for parked units, which is the maximum number of units that can be temporarily held at the station. 
As shown in Figure~\ref{fig:modular bidirectional illustration}, units released by one service can be stored at the station and reused by a later service in either direction.
Distances between adjacent stations and station dwell times are known at the planning stage. The corridor is a dedicated right-of-way with minimal external interference, which supports stable operations. We therefore assume a constant travel speed and deterministic travel times between adjacent stations, known at the planning stage.

\subsection{Modular Vehicle}

A modular vehicle consists of multiple self-propelled units that can operate either jointly or independently. Unlike conventional bus systems, modular services are not restricted to terminal-to-terminal trips: vehicles may depart from intermediate stations when units are available and may also terminate at non-terminal stations.

All units are physically homogeneous and share identical technical characteristics. Once docked, each unit is assigned either a passenger or freight function, which remains unchanged until the unit is undocked. Consequently, passenger and freight services are not mixed within a single unit. Vehicle capacity depends on both the number of docked units and their functional assignment. To ensure operational feasibility, the total vehicle length is limited by an upper bound $l$.

Docking and undocking enable flexible vehicle composition and en-route capacity adjustment, but they require non-negligible processing times that affect unit circulation and departure feasibility. These times are often assumed to be instantaneous or absorbed into dwell times in the literature, which can lead to infeasible schedules in practice. In this study, docking and undocking times are modeled explicitly to ensure temporal feasibility.

\subsection{Passenger Demand Group}

Passenger demand is represented by a set of groups \(P\). Each group \(p \in P\) consists of passengers sharing the same origin station \(o_p \in S\), destination station \(\delta_p \in S\), arrival time \(t_p\), and delivery time window \([a_p, b_p]\). 
Here, \(t_p\) denotes the time when passengers become available for boarding at \(o_p\), and \([a_p, b_p]\) specifies the required arrival window at \(\delta_p\).
The number of passengers in group \(p\) is uncertain and is modeled as a random variable \(\tilde{q}_p\). 
Because each group represents aggregated demand over a service-related time interval rather than individual passenger arrivals, we adopt a normal approximation, i.e., \(\tilde{q}_p \sim \mathcal{N}(\mu_p, \sigma_p^2)\) \citep{liang2023online}. 

Each passenger group \(p\) has a travel direction \(k_p\), either upbound or downbound. Passengers are assumed to remain on the same vehicle once boarded, which is consistent with direct corridor trips. Within a vehicle, passengers can move between docked units through inter-unit passageways en route. If needed, they may be guided (e.g., via a mobile application) to relocate across units so that undocking at intermediate stations does not interrupt their trips. Accordingly, passenger capacity is modeled at the vehicle level.

Passenger assignment is not restricted to a first-come-first-served rule. Given flexible routes and limited capacity, a single departure may not serve all waiting passengers. Demand from the same group may be split across multiple vehicles, which helps balance system utilization. Passenger assignment is determined by the delivery time windows \([a_p, b_p]\), and total waiting time is included in the objective.

\subsection{Freight Demand Group}

Freight demand is represented by a set \(F\). Each request \(f \in F\) is characterized by an origin station \(o_f \in S\), a destination station \(\delta_f \in S\), a deterministic demand quantity \(q_f\), a delivery time window \([a_f, b_f]\), and a travel direction \(k_f\). 
Freight is standardized and measured in discrete boxes. Accordingly, \(q_f\) denotes the number of boxes to be transported for request \(f\).
The time window \([a_f, b_f]\) specifies that service cannot start before \(a_f\) and must be completed no later than \(b_f\). 
Freight demand may be split and need not be fully satisfied, since passenger service is given priority and any unmet quantity incurs a penalty in the objective.

Freight handling follows a unit-based design. Most existing studies assume item-level handling within short dwell times. However, such operations may be difficult in practice and may interfere with passenger boarding and alighting. In this study, freight loading and unloading are performed only at the unit level and only at stations. Boxes are packed into dedicated freight units during or before the docking process, after which these units are attached to vehicles for transportation along the corridor. Similarly, boxes can be unloaded only when the corresponding freight unit is undocked at a station.

Freight transfers are allowed at en-route stations. 
A freight transfer refers to unloading freight from one freight unit and reloading it onto a freight unit of a subsequent service for continued transportation. 
During this process, freight may be temporarily stored at the station and consolidated with other freight. 
To avoid detours and improve capacity utilization, freight transfers are restricted to services traveling in the same direction. 

\subsection{Optimization Problem}

\begin{figure}[t]
\centering
\includegraphics[width=0.9\linewidth]{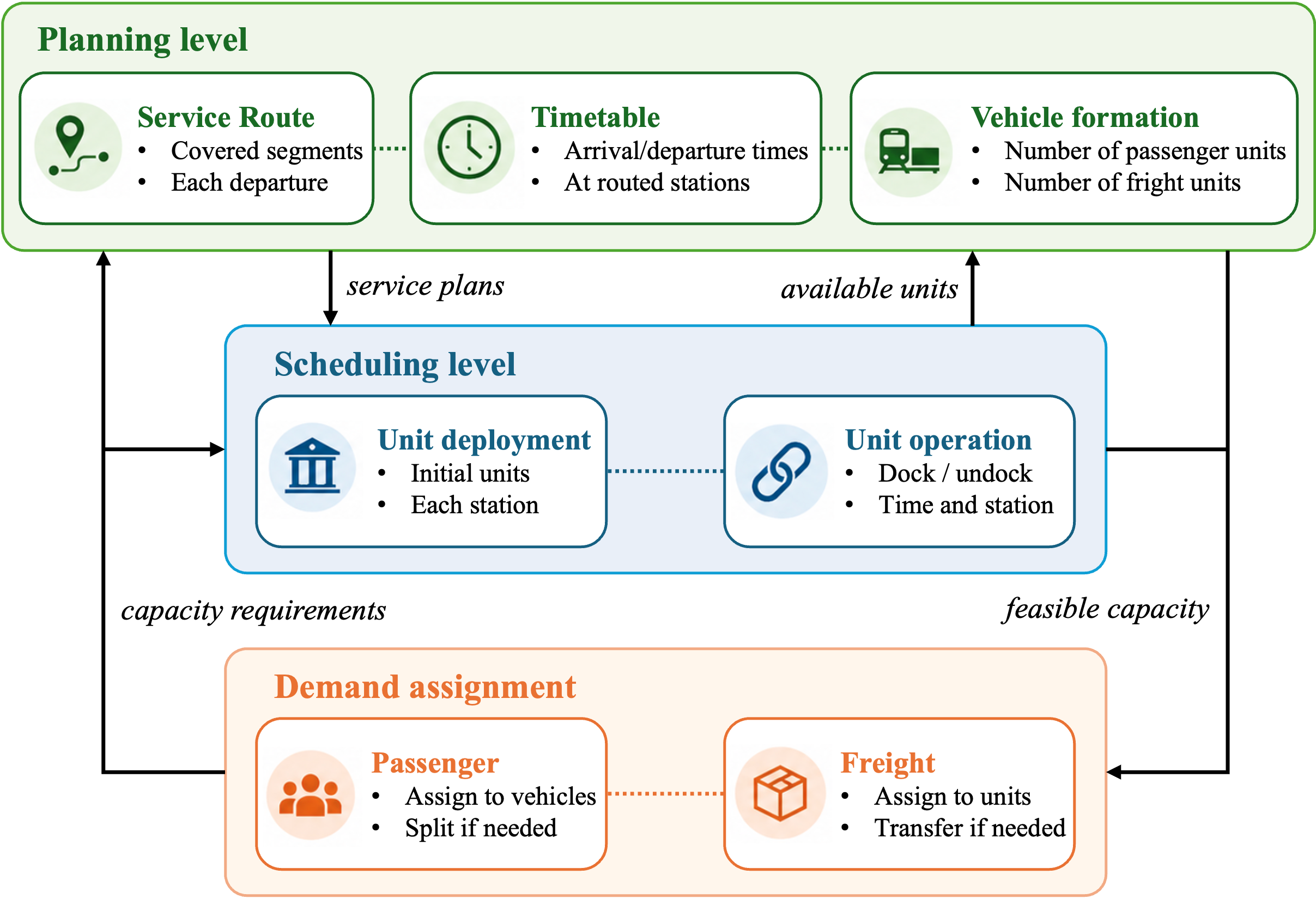}
\caption{Decision overview of the modular passenger--freight planning and scheduling problem.}
\label{fig:decision overview}
\end{figure}

Given the setting above, the optimization problem jointly determines the planning and scheduling of modular vehicles in a bidirectional corridor setting. Figure~\ref{fig:decision overview} provides an overview of the decision structure, which consists of four main decision components:

\begin{itemize}

    \item \textbf{Service plan for modular vehicles.}
    This component determines the service route of each vehicle along the corridor, the station departure timetable subject to minimum headway constraints, and the vehicle composition on each route segment in terms of the number of docked units and their passenger or freight functions.
    
    \item \textbf{Unit scheduling.}
    The service plan is implemented through detailed schedules for individual units. For each unit, this component determines its sequence of operations over time and space, including when and where it is docked to or undocked from a vehicle, its assigned function (passenger or freight) while attached to a vehicle, and whether it is assigned to successive services or remains idle in the station pool.

    \item \textbf{Passenger flow assignment.}
    Passenger demand is assigned to vehicle departures to satisfy delivery time windows. Passengers remain on the same vehicle once boarded, but may relocate across units within a vehicle through inter-unit passageways. Demand from a group may be split across multiple departures to accommodate capacity constraints and improve system efficiency.

    \item \textbf{Freight flow assignment.}
    Freight demand is assigned and handled at the unit level. Loading and unloading are coordinated with docking and undocking operations at stations. Freight transfers between services are allowed at intermediate stations, but are restricted to services traveling in the same direction to avoid detours. Any unmet demand incurs a penalty cost in the objective.

\end{itemize}

\section{Network Construction}
\label{Network Construction}

The planning and scheduling of modular vehicles, together with the assignment of passenger and freight flows are coupled across space, time, and operational states. To capture these interdependencies within a unified optimization framework, we construct a space--time--state network. For clarity, this section first introduces the underlying space--time network as the foundation and then extends it with operational states to construct the complete space--time--state network. Finally, we define the flow representations and show how key system processes are modeled as network flows.

\subsection{Initial Space--Time Network}
The planning horizon $[\underline{t},\overline{t}]$, with duration $H=\overline{t}-\underline{t}$, is discretized into minutes and represented by the time set $T$, where $\underline{t}$ and $\overline{t}$ denote its start and end, respectively. 
The space--time network is defined as $G=(N,E)$, where each node $(s,t)$ represents station $s \in S$ at time $t \in T$.

Arcs represent feasible unit trajectories and are divided into service arcs and storage arcs. A service arc connects $(s,t)$ to $(s',t')$ if $s \neq s'$, $s$ and $s'$ are adjacent stations, $t' > t$, and $t'-t$ equals the fixed travel time from $s$ to $s'$ plus the dwell time at $s'$. Since travel speeds and dwell times are fixed, all feasible movements between adjacent stations can be enumerated in advance, and the corresponding service arcs can be constructed over the time indices in $T$. Service arcs are further classified as upbound or downbound according to travel direction.

Passenger and freight movements are modeled in a unified arc framework by indexing arcs by units. Specifically, we define \(U=\{u_0\}\cup U_f\), where \(u_0\) represents all coupled passenger units as one virtual unit because they form a shared passenger space within the vehicle. The set $U_f$ denotes potential freight units, and its size is bounded by the maximum number of freight-dedicated units allowed in a vehicle.

For each feasible movement between adjacent stations and each departure time $t \in T$, we construct one service arc for $u_0$ to represent passenger service and one service arc for each $u \in U_f$ to represent freight service at the unit level. These parallel arcs provide the network basis for representing vehicle composition.

\begin{figure}[t]
  \centering
  \begin{subfigure}{0.32\textwidth}
    \centering
    \includegraphics[width=1.1\linewidth]{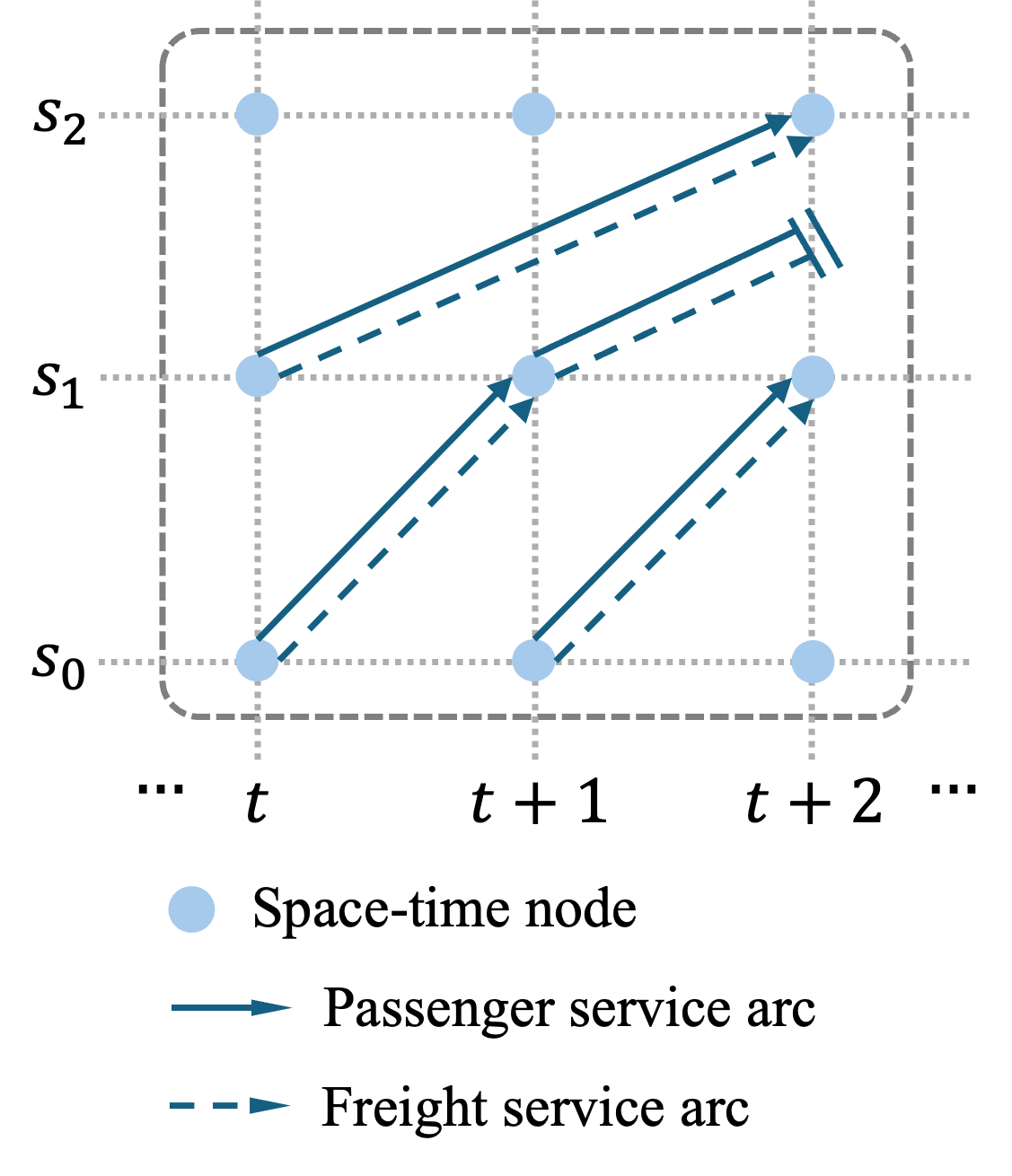}
    \caption{Upbound Service}
    \label{fig:upbound_network}
  \end{subfigure}
  \hfill
  \begin{subfigure}{0.32\textwidth}
    \centering
    \includegraphics[width=1.1\linewidth]{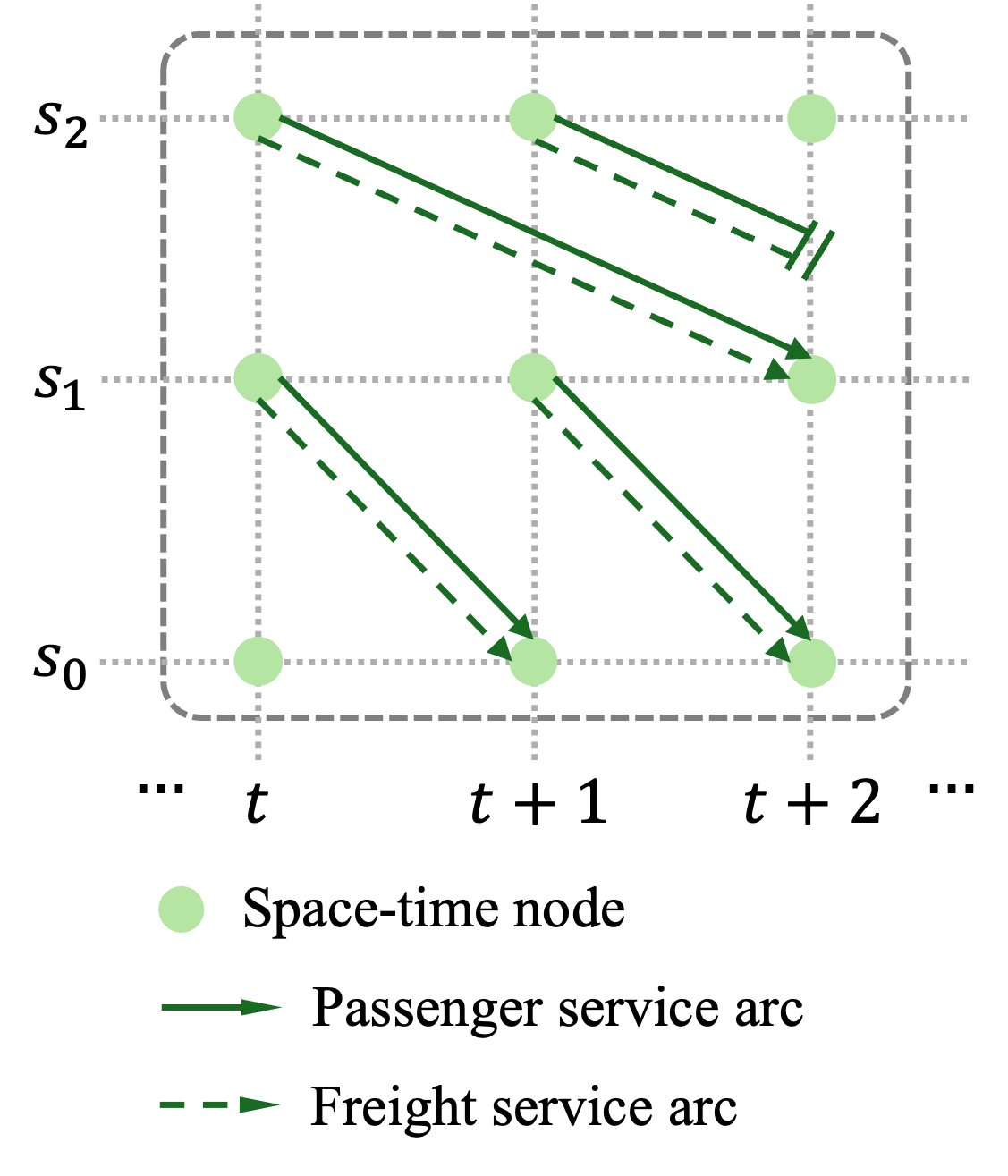}
    \caption{Downbound Service}
    \label{fig:downbound_network}
  \end{subfigure}
  \hfill
  \begin{subfigure}{0.32\textwidth}
    \centering
    \includegraphics[width=1.1\linewidth]{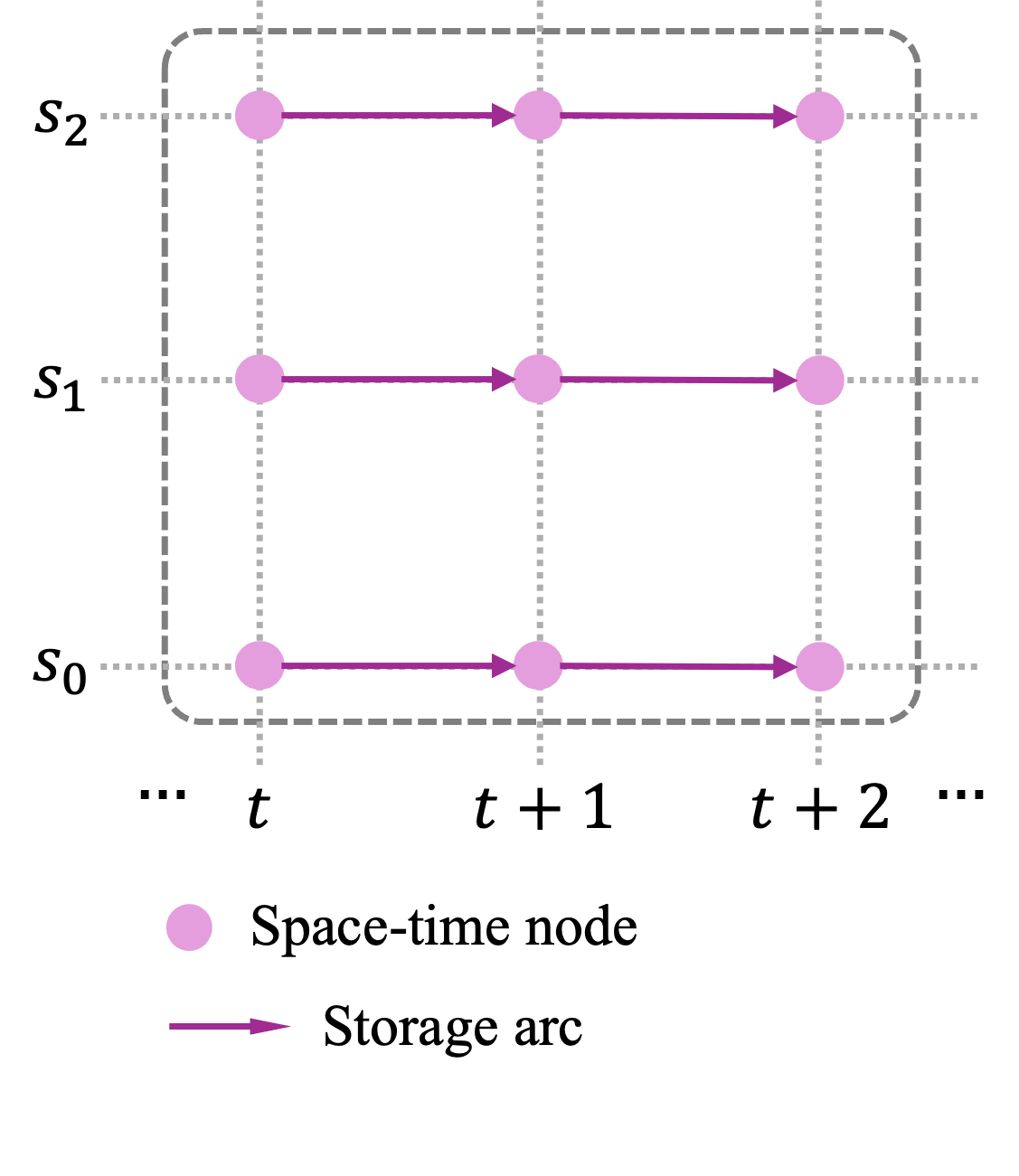}
    \caption{Storage}
    \label{fig:storage st network}
  \end{subfigure}
  \caption{Illustration space-time network for service and storage arcs.}
  \label{fig:st network}
\end{figure}

To illustrate the construction of the space--time network, Figure~\ref{fig:st network} presents a simple example with three stations and three consecutive time points. The travel time, including dwell time, is 1 minute between $s_0$ and $s_1$, and 2 minutes between $s_1$ and $s_2$. Each node represents a station--time pair. Figures~\ref{fig:upbound_network} and~\ref{fig:downbound_network} show the service arcs for the upbound and downbound directions, respectively. In this example, each vehicle can include at most one freight-dedicated unit, so each feasible movement is represented by one passenger service arc and one freight service arc.

Storage arcs represent units idling at stations. A storage arc connects consecutive nodes $(s,t)$ and $(s,t+1)$ at the same station, indicating that a unit remains in the station pool during the interval $[t,t+1)$. Unlike service arcs, storage arcs do not distinguish functional roles: once a unit is undocked, it is treated as a homogeneous resource regardless of its previous assignment and can be reassigned to either passenger or freight service in a later departure as needed. If heterogeneous units are considered, storage arcs can be defined separately for passenger and freight units in the same way as service arcs.

Using the same example, storage arcs are added at each station to complete the space--time network, as shown in Figure~\ref{fig:storage st network}. 
These arcs represent units staying in the station pool, with consecutive storage arcs capturing continuous idling at the same station.

\subsection{Space--Time--State Network}

We define the state set as $K=\{k_1,k_2,k_3\}$, where each modular unit can be in one of three states: upbound ($k_1$), storage ($k_2$), or downbound ($k_3$). To capture transitions between these states, we extend the initial space--time network to a space--time--state network $\mathcal{G}=(\mathcal{N},\mathcal{E})$. In this network, each node $(s,t,k)$ represents a unit at station $s \in S$ and time $t \in T$ in state $k \in K$. For clarity, we denote by $\mathcal{N}_1$, $\mathcal{N}_2$, and $\mathcal{N}_3$ the subsets of nodes corresponding to the upbound, storage, and downbound states, respectively.

Arcs in the space--time--state network fall into three categories. Service arcs are defined on the upbound ($k_1$) and downbound ($k_3$) planes and are inherited from the initial space--time network, representing movements between adjacent stations in the corresponding direction. Storage arcs are defined on the storage plane ($k_2$) and connect consecutive time nodes at the same station, representing units remaining in the station pool.

To model docking and undocking, we introduce reconfiguration arcs that connect nodes across different states. At station $s \in S$, an arc from $(s,t,k_1)$ or $(s,t,k_3)$ to $(s,t+\tau^{ud},k_2)$ represents an undocking operation in which units leave a vehicle and enter the station pool. Conversely, an arc from $(s,t,k_2)$ to $(s,t+\tau^{dk},k_1)$ or $(s,t+\tau^{dk},k_3)$ represents a docking operation in which units leave the station pool to join a vehicle. The parameters $\tau^{dk}$ and $\tau^{ud}$ denote the docking and undocking process times, respectively, and are captured by the time spans of the corresponding arcs.

Similar to service arcs, reconfiguration arcs are indexed by unit type to enforce consistent vehicle composition. For the virtual passenger unit $u_0$, each operation is represented by a single aggregated arc, whereas for freight units $u \in U_f$, separate arcs are constructed to capture unit-level operations.

\begin{figure}[t]
\centering
\begin{subfigure}{0.32\textwidth}
    \centering
    \includegraphics[width=1.1\linewidth]{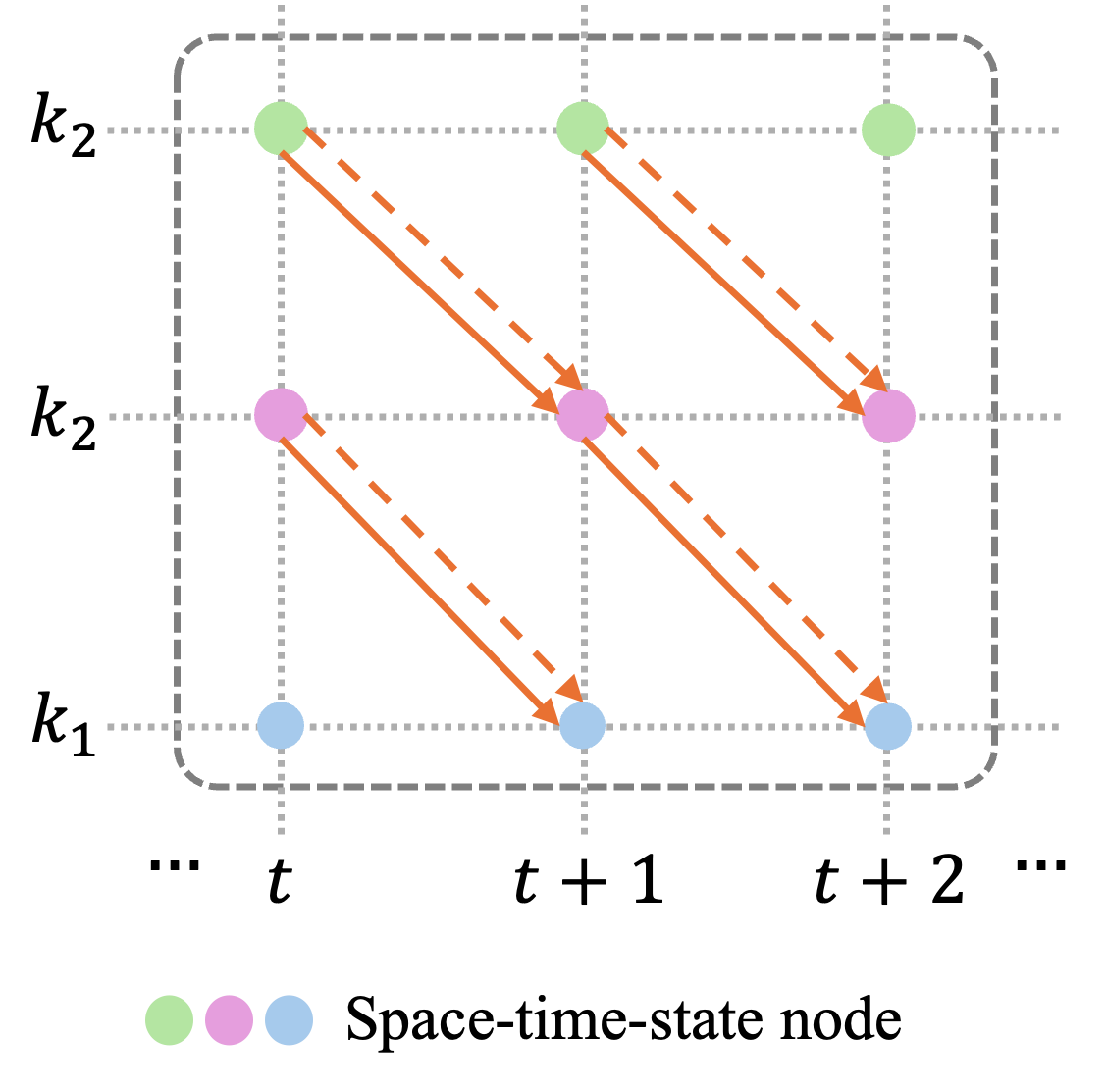}
    \caption{Terminal station $s_0$}
    \label{fig:Terminal Station s1}
\end{subfigure}
\hfill
\begin{subfigure}{0.32\textwidth}
    \centering
    \includegraphics[width=1.1\linewidth]{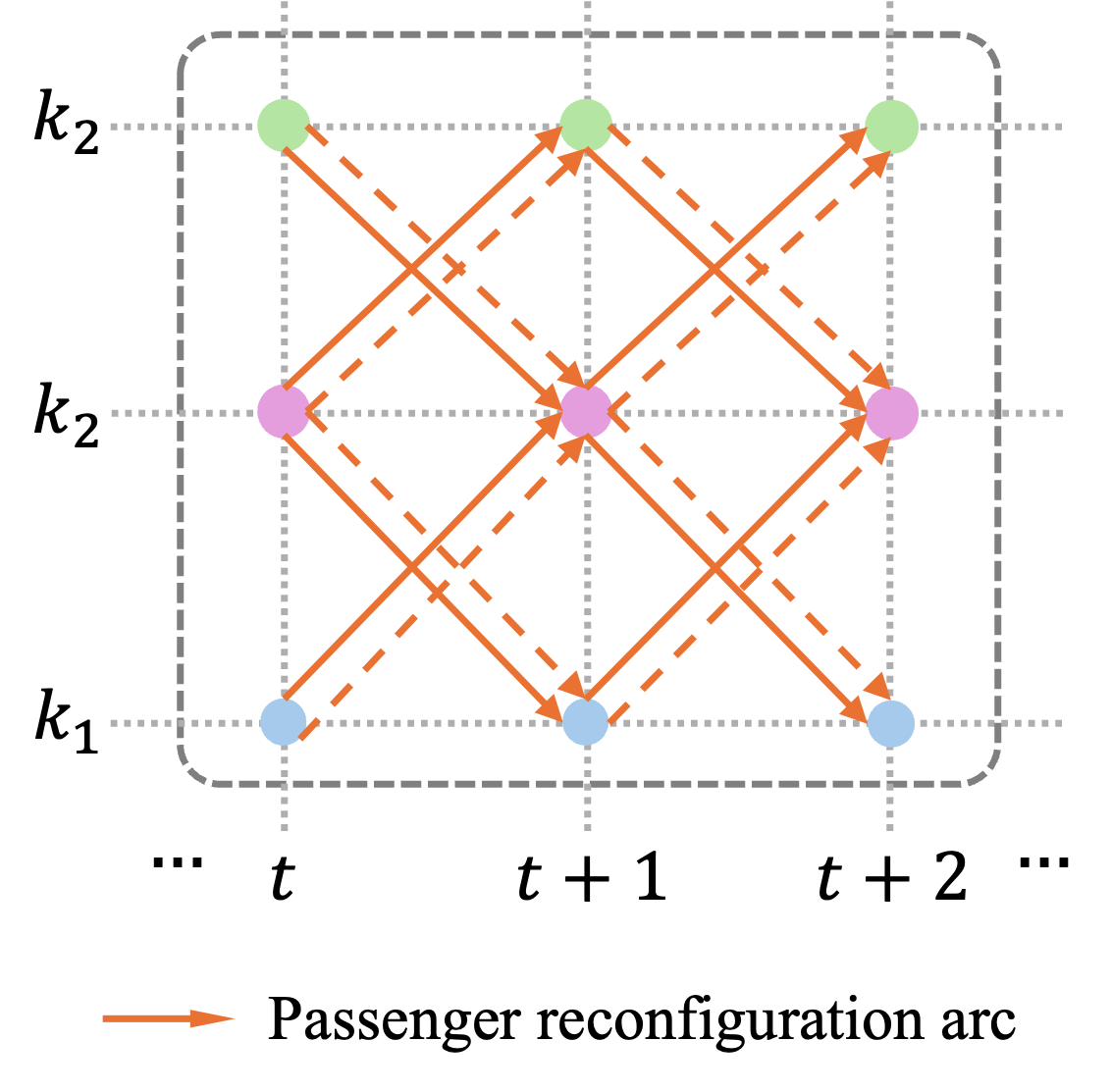}
    \caption{Intermediate station $s_1$}
    \label{fig:Intermediate Station s2}
\end{subfigure}
\hfill 
\begin{subfigure}{0.32\textwidth}
    \centering
    \includegraphics[width=1.1\linewidth]{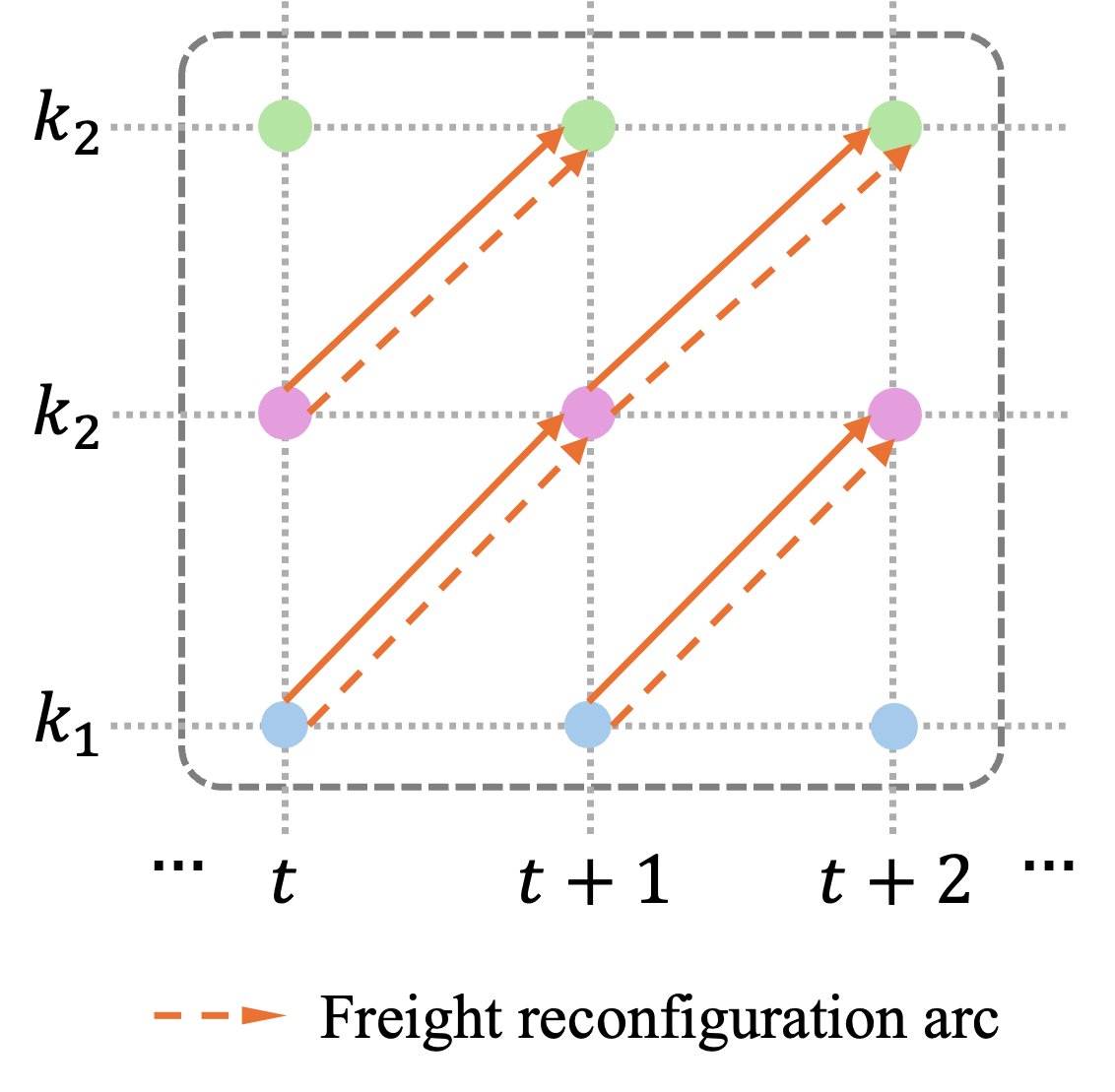}
    \caption{Terminal station $s_2$}
    \label{fig:Terminal Station s3}
\end{subfigure}

\caption{Illustration of reconfiguration arc construction at intermediate and terminal stations.}
\label{fig:reconfiguration arcs}
\end{figure}

Recall the corridor with three stations $s_0$, $s_1$, and $s_2$, where $s_0$ and $s_2$ are terminal stations and $s_1$ is an intermediate station. Introducing the state dimension unifies the two directional space--time networks into a single space--time--state network. With 1-minute docking and undocking times, reconfiguration arcs are added to represent feasible unit transitions at stations. Figure~\ref{fig:reconfiguration arcs} illustrates the construction, where solid and dashed arcs denote passenger and freight reconfiguration arcs, respectively.

At the intermediate station $s_1$, units can be undocked from arriving services (in either direction) into the storage state, and stored units can be docked to departures in either direction. We therefore add reconfiguration arcs between the corresponding station--time--state nodes. At terminal stations, reconfiguration is one-sided: at $s_0$, arriving downbound units enter storage, and any departing units can only join upbound services; at $s_2$, arriving upbound units enter storage, and any departing units can only join downbound services.

For each station, the storage nodes at times $\underline{t}$ and $\overline{t}$ define the start and end of unit schedules. To maintain feasibility at the boundaries of the planning horizon, we add auxiliary connections when a docking or undocking operation would otherwise extend beyond the horizon. If docking would start before $\underline{t}$, the departure node is connected directly to the storage node at $\underline{t}$. If undocking would end after $\overline{t}$, the arrival node is connected directly to the storage node at $\overline{t}$.

Each arc $e \in \mathcal{E}$ is associated with a travel or processing time $\tau_e$, which determines the time difference between its start and end nodes. For convenience, we define several arc subsets. For each $u \in U$, let $\mathcal{E}_u$ denote the set of service and reconfiguration arcs associated with $u$. 
Let $\mathcal{E}_*$ denote the set of service arcs, $\mathcal{E}_\circ$ the set of storage arcs, $\mathcal{E}_\triangle$ the set of reconfiguration arcs, and $\mathcal{E}_\circ(s)$ the set of storage arcs at station $s$.
For any node $(s,t,k) \in \mathcal{N}$, let $\mathcal{E}^+(s,t,k)$ and $\mathcal{E}^-(s,t,k)$ denote the sets of arcs leaving and entering the node, respectively. For any subset $\mathcal{E}' \subseteq \mathcal{E}$, define $\mathcal{E}'^{\pm}(s,t,k)=\mathcal{E}^{\pm}(s,t,k)\cap \mathcal{E}'$.

\subsection{Flow Representation and Coupling}

This section introduces unit, passenger, and freight flows on the space--time--state network and describes how they are coupled to ensure consistent and feasible operations.

\subsubsection{Vehicle and Unit Flows}

For each station and time node on the upbound and downbound planes, we decide whether a service departs to the next station. If a departure is scheduled, its composition on that segment is determined by the outgoing service arcs, namely, (i) the number of units assigned to the passenger service arc and (ii) the activation of freight service arcs, where each activated freight arc indicates the inclusion of one freight-dedicated unit. Each service route is defined by a sequence of consecutive service arcs, and the associated unit flows describe how the vehicle composition evolves along the route.

Given these service decisions, unit flows describe how individual units move through the space--time--state network. They ensure that every activated service arc is supported by sufficient units and that flow is conserved at each node. Unit flows also track unit circulation across stations and time, thereby ensuring that all departures are feasible under a limited fleet.
At the initial time $\underline{t}$, each station $s \in S$ has a storage node representing the initial distribution of units. At the terminal time $\overline{t}$, storage nodes represent the final distribution. By enforcing identical initial and final distributions, we can obtain a periodic schedule over the planning horizon.

\begin{figure}[ht]
\centering
\includegraphics[width=0.6\linewidth]{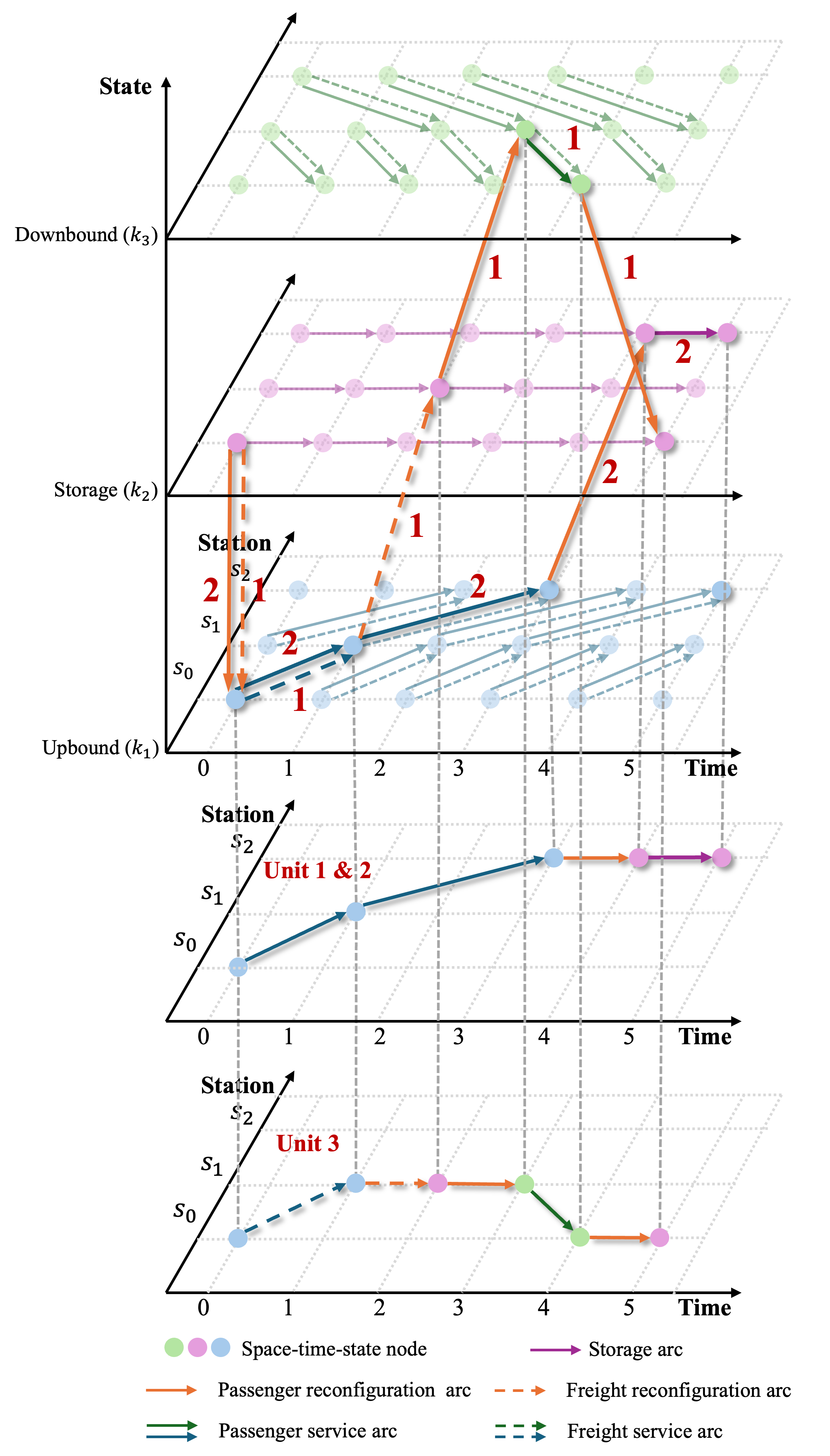}
\caption{Example of modular vehicle scheduling in proposed space-time-state network.}
\label{fig:units scheduling example}
\end{figure}

Figure~\ref{fig:units scheduling example} illustrates a modular vehicle schedule on the space--time--state network. On the upbound plane, a vehicle departs from station $s_0$ at time $t=0$ and travels toward $s_2$. Its composition on segment $(s_0,s_1)$ includes two passenger units and one freight unit, while on segment $(s_1,s_2)$ only the passenger units remain coupled. On the downbound plane, another vehicle departs from $s_1$ at time $t=3$ toward $s_0$ with a single passenger unit.
To implement this service plan, three units are scheduled. Units~1 and~2 depart from $s_0$ at $t=0$ as passenger units, travel with the upbound vehicle to $s_2$, and then remain idle. Unit~3 is loaded with freight at $s_0$, travels as a freight unit to $s_1$, unloads upon arrival, and is then reassigned as a passenger unit to serve the downbound vehicle to $s_0$.

\subsubsection{Passenger Flows}

Passenger assignment is modeled as a multi-commodity flow problem on the space--time--state network. Each demand \(p \in P\) is treated as an aggregate passenger-flow commodity, and the assigned flow must be supported by the available passenger capacity on the corresponding vehicle departures. 

To initialize and terminate flows, we introduce auxiliary nodes and arcs. For each demand \(p \in P\), a source node \(o'_p\) is connected to all feasible nodes \((o_p,t,k_p)\) with \(t \geq t_p\). Each such arc represents a possible boarding time, and its travel time accounts for passenger waiting. Similarly, a sink node \(\delta'_p\) is connected from all nodes \((\delta_p,t,k_p)\) with \(t \in [a_p,b_p]\) via zero-cost arcs. 

For each \(p \in P\), let \(\mathcal{E}_p\) denote the set of admissible arcs for assigning the flow of \(p\). This set includes all passenger service arcs on the state plane \(k_p\), as well as the auxiliary arcs incident to \(o'_p\) and \(\delta'_p\).

\subsubsection{Freight Flows}

Freight demand allocation is formulated as a multi-commodity flow problem on the space--time--state network, where each request $f \in F$ is treated as a distinct commodity. Freight is handled at the unit level and may be partially served, with any unmet quantity penalized in the objective. 

For each request \(f \in F\), the flow starts from the storage node \((o_f,a_f,k_2)\), where freight becomes available, and ends at \((\delta_f,b_f,k_2)\), ensuring arrival at destination \(\delta_f\) no later than \(b_f\). Freight reconfiguration arcs connect the storage state $k_2$ and the corresponding service state $k_f$. Loading occurs when or before a freight unit docks to a vehicle, and unloading occurs when or after the freight unit undocks at a station. Once freight is loaded onto a freight unit and the unit is docked to a vehicle, the freight flow follows the sequence of service arcs associated with that unit until the unit is undocked. If the unit is undocked at an intermediate station before the freight reaches its destination, the freight is unloaded to the station and can be transferred to a later service by being reloaded onto another unit.

For each $f \in F$, let $\mathcal{E}_f$ denote the set of admissible arcs for assigning freight flow. This set includes all storage arcs, all freight service arcs in direction $k_f$, and all reconfiguration arcs connecting the storage state $k_2$ with the corresponding service state.

\subsubsection{Flow Coupling}

Flow coupling links demand flows with unit flows through capacity constraints on the space--time--state network. Passenger flows are coupled at the vehicle level, whereas freight flows are coupled at the unit level.

For passenger flows, all coupled passenger units in a vehicle are treated as a single network entity. 
A passenger service arc therefore represents their aggregated capacity on that segment. 
Let $c_p$ denote the capacity of one passenger unit. 
The arc capacity is determined by the number of assigned passenger units, and the passenger flow on the arc cannot exceed this capacity.

For freight flows, each freight service arc and each reconfiguration arc corresponds to a specific freight-dedicated unit. Let $c_f$ denote the loading capacity of a single freight unit. The freight flow on each such arc must not exceed $c_f$, ensuring that loading, unloading, and transfer operations remain consistent with unit-based handling.
\section{Model Formulation}
\label{Model Formulation}

Based on the space--time--state network and flow definitions, we formulate the problem as a stochastic mixed-integer programming (MIP) model. 
The main modeling novelty lies in constructing an integrated service--unit--flow network structure that links the three decision layers shown in Figure~\ref{fig:decision overview}: vehicle service planning, unit scheduling, and demand assignment. 
This structure exposes the key couplings of the problem, including the link between service planning and unit-scheduling feasibility, and the interaction between passenger-priority service provision and freight use of modular capacity. 
These couplings motivate the decomposition-based exact and heuristic solution methods developed later. 
To handle stochastic passenger demand, we adopt a chance-constrained approach and reformulate the resulting model as a mixed-integer linear programming (MILP) model. For readability, the main notation used in the formulation is summarized in Appendix~\ref{apx:notation}.

\subsection{Stochastic MIP Formulation}

We introduce the following decision variables:

\begin{itemize}
    \item \textbf{Departure decisions.}
    Binary variables $x_{(s,t,k)} \in \{0,1\}$ for $(s,t,k) \in \mathcal{N}_1 \cup \mathcal{N}_3$ indicate whether a vehicle departure is scheduled at station $s$ and time $t$ in direction $k$.

    \item \textbf{Initial fleet allocation.}
    Non-negative integer variables $n_s \in \mathbb{Z}_{\ge 0}$ for $s \in S$ specify the number of units initially available at station $s$ at time $\underline{t}$. These variables represent the fleet required to support all scheduled operations over the planning horizon.

    \item \textbf{Vehicle configuration and unit scheduling.}
    Non-negative integer variables $y_e \in \mathbb{Z}_{\ge 0}$ for $e \in \mathcal{E}$ denote the number of units assigned to arc $e$, capturing vehicle composition, storage, and unit movements in the network.

    \item \textbf{Passenger flow assignment.}
    Non-negative continuous variables $z_e^p \ge 0$ for $p \in P$ and $e \in \mathcal{E}_p$ denote the amount of passenger demand $p$ routed through arc $e$.

    \item \textbf{Freight flow assignment.}
    Non-negative continuous variables $z_e^f \ge 0$ for $f \in F$ and $e \in \mathcal{E}_f$ denote the amount of freight demand $f$ routed through arc $e$.
\end{itemize}

The objective is to minimize the total system cost over the planning horizon:
\begin{equation}
    \min \quad 
    w_1 \sum_{s \in S} n_s
    + w_2 \sum_{p \in P} \sum_{e \in \mathcal{E}^{+}(o'_p)} \tau_e z_e^p 
    + w_3 \sum_{f \in F} \left(q_f - \sum_{e \in \mathcal{E}^{+}(o_f, a_f, k_2)} z_e^f \right).
\end{equation}

The first term represents fleet operating cost, where $n_s$ is the number of units initially deployed at station $s$, and $w_1$ is the cost per unit over the planning horizon. The second term represents passenger waiting cost, where waiting time is captured by the auxiliary arcs with time span $\tau_e$ and $w_2$ denotes the value of time. The third term penalizes unmet freight demand, where $q_f$ is the total demand of request $f$, the routed flow represents the served portion, and $w_3$ is the per-unit penalty cost.

\textbf{Service planning and unit scheduling constraints.}
Constraint~\eqref{c1} enforces a minimum headway $\underline{h}$ between consecutive departures at each station in each direction. Constraints~\eqref{c2} and~\eqref{c4} link the initial fleet allocation with unit dispatching at $\underline{t}$ and unit returning at $\bar{t}$, respectively, thereby imposing a periodic unit-scheduling condition at each station. Constraint~\eqref{c3} bounds the initial allocation at each station by its storage capacity $\overline{n}_s$. Constraint~\eqref{c5} enforces flow conservation at storage nodes for all intermediate times, maintaining the balance of units in the station pool.
Constraint~\eqref{c6} enforces unit-specific flow conservation at service nodes, ensuring that vehicle configuration is preserved along consecutive service arcs and can change only through docking or undocking operations. Constraint~\eqref{c7} restricts each freight-related arc to at most one unit. Constraint~\eqref{c8} links departure decisions with vehicle composition: if a departure is scheduled, the number of units in the departing vehicle must be between 1 and $l$; otherwise, no units are assigned. Under the standard convention that an empty summation equals zero, nodes without outgoing service arcs are automatically prevented from generating departures. Constraint~\eqref{c9} limits the number of stored units at each station by its storage capacity. Constraints~\eqref{c10}--\eqref{c12} define the variable domains.
\begin{align}
    &\sum_{t' \in T : \, t \le t' \le t+\underline{h}-1} x_{(s,t',k)} \leq 1, && \forall (s,t,k) \in \mathcal{N}_1 \cup \mathcal{N}_3 \label{c1} \\ 
    &\sum_{e \in \mathcal{E}^{+}(s,\underline{t},k_2)} y_e = n_s, && \forall s \in S, \label{c2} \\
    &\sum_{e \in \mathcal{E}^{-}(s,\overline{t},k_2)} y_e = n_s, && \forall s \in S, \label{c4} \\
    &n_s \le \overline{n}_s, && \forall s \in S,\label{c3}  \\
    &\sum_{e \in \mathcal{E}^{+}(s,t,k)} y_e = \sum_{e \in \mathcal{E}^{-}(s,t,k)} y_e, && \forall (s,t,k) \in \mathcal{N}_2, \; t \neq \underline{t}, \; t \neq \overline{t},  \label{c5} \\
    &\sum_{e \in \mathcal{E}_u^{+}(s,t,k)} y_e = \sum_{e \in \mathcal{E}_u^{-}(s,t,k)} y_e, && \forall (s,t,k) \in \mathcal{N}_1 \cup \mathcal{N}_3, u \in U, \label{c6} \\
    &y_e \le 1 && \forall e \in \mathcal{E}_u, u \in U_f, \label{c7} \\
    &x_{(s, t, k)} \le \sum_{e \in\mathcal{E}_*^{+}(s, t, k)} y_e \le l x_{(s, t, k)}, && \forall (s,t,k) \in \mathcal{N}_1 \cup \mathcal{N}_3,\label{c8} \\
    &y_e \le \overline{n}_s, && \forall s \in S, e \in \mathcal{E}_\circ(s),\label{c9} \\
    &x_{(s,t,k)} \in \{0,1\}, && \forall (s,t,k) \in \mathcal{N}_1 \cup \mathcal{N}_3, \label{c10} \\
    &n_s \in \mathbb{Z}_{\geq 0} && \forall s \in S, \label{c11} \\
    &y_e \in \mathbb{Z}_{\geq 0} && \forall e \in \mathcal{E}. \label{c12}
 \end{align}

\textbf{Passenger flow distribution constraints.}
Constraint~\eqref{c13} defines the passenger demand satisfaction requirement: for each group $p \in P$, the assigned flow leaving the origin and entering the destination must match the stochastic demand $\tilde{q}_p$. 
This stochastic constraint is reformulated later through a chance-constrained formulation.
Constraint~\eqref{c14} imposes flow conservation at all intermediate nodes in the passenger state plane. Constraint~\eqref{c15} enforces arc-capacity limits, requiring that the passenger flow on each passenger service arc does not exceed the capacity induced by the assigned passenger units. Constraint~\eqref{c17} defines the non-negativity of passenger flow variables.
\begin{align}
    &\sum_{e \in \mathcal{E}^{+}(o'_p)} z_e^p =  \sum_{e \in \mathcal{E}^{-}(\delta'_p)} z_e^p  = \tilde{q}_p, && \forall p \in P, \label{c13} \\
    &\sum_{e \in \mathcal{E}^{+}(s,t,k_p)} z_e^p = \sum_{e \in \mathcal{E}^{-}(s,t,k_p)} z_e^p, && \forall p \in P, (s,t,k_p) \in \mathcal{N}, \label{c14} \\
    &\sum_{p \in P} z_e^p \le c_p y_e, && \forall e \in \mathcal{E}_{u_0}, \label{c15} \\
    &z_e^p \ge 0, && \forall p \in P, e \in \mathcal{E}_p. \label{c17}
\end{align}

\textbf{Freight flow distribution constraints.}
Constraint~(\ref{c18}) limits served freight by requiring that, for each request $f \in F$, the total flow from the origin to the destination does not exceed the demand $q_f$, thereby allowing partial service when necessary. Constraint~(\ref{c19}) imposes flow conservation at storage-state nodes, ensuring that freight is neither created nor lost while waiting, transferring, or being reallocated at stations. Constraint~(\ref{c20}) enforces flow balance on the freight service plane, requiring that once freight is assigned to a freight unit $u \in U_f$, it stays with that unit until the unit is undocked. Constraint~(\ref{c21}) enforces capacity limits, requiring that the freight flow on each unit-specific arc does not exceed the unit capacity $c_f$ when the unit is assigned. Constraint~(\ref{c23}) defines the non-negativity of freight flow variables.
\begin{align} 
&\sum_{e \in \mathcal{E}^{+}(o_f, a_f, k_2)} z_e^f = \sum_{e \in \mathcal{E}^{-}(\delta_f, b_f, k_2)} z_e^f \le q_f, && \forall f \in F, \label{c18} \\ 
&\sum_{e \in \mathcal{E}^{+}(s,t,k)} z_e^f = \sum_{e \in \mathcal{E}^{-}(s,t,k)} z_e^f, && \forall f \in F, (s,t,k) \in \mathcal{N}_2, \label{c19} \\ 
&\sum_{e \in \mathcal{E}_u^{+}(s,t,k_f)} z_e^f = \sum_{e \in \mathcal{E}_u^{-}(s,t,k_f)} z_e^f, && \forall f \in F, (s,t,k_f) \in \mathcal{N}, u \in U_f, \label{c20} \\ 
&\sum_{f \in F} z_e^f \le c_f y_e, && \forall e \in \mathcal{E}_u, u \in U_f, \label{c21} \\ 
&z_e^f \ge 0, && \forall f \in F, e \in \mathcal{E}_f. \label{c23} 
\end{align}

Because passenger and freight flow variables are continuous, fractional flows may arise when aggregate demand is assigned across feasible services.
However, unnecessary splitting is discouraged by path-specific passenger waiting costs and restricted by service feasibility, unit capacities, time windows, transfer feasibility, and penalties for unserved freight demand.
This aggregate-flow treatment is consistent with tactical service network design and capacity planning, where the focus is on service planning and resource allocation rather than  dispatching decisions \citep{pei2021vehicle, zhang2025robust}.

\subsection{Model Reformulation}
\label{Model reFormulation}

In this system, passengers are assumed to travel directly without transfers. As a result, each passenger group can be restricted to a limited set of feasible paths, and each passenger service arc belongs to at most one path for a given demand. This property allows the arc-based passenger assignment formulation to be reformulated in a path-based form, which removes the flow balance constraints for intermediate nodes.

For each passenger group \(p \in P\), let \(R_p\) denote the set of feasible paths from \(o'_p\) to \(\delta'_p\), where each path satisfies the boarding-time and service-time-window requirements of group \(p\). Let \(z_r^p \ge 0\) denote the flow assigned to path \(r \in R_p\). The objective becomes
\begin{equation}
    \min \quad 
    w_1 \sum_{s \in S} n_s
    + w_2 \sum_{p \in P} \sum_{r \in R_p} \tau_r z_r^p
    + w_3 \sum_{f \in F} \left(q_f - \sum_{e \in \mathcal{E}^{+}(o_f,a_f,k_2)} z_e^f \right).
\end{equation}

Accordingly, the passenger assignment constraints~\eqref{c13}--\eqref{c17} are replaced by
\begin{align}
    &\sum_{r \in R_p} z_r^p = \tilde{q}_p, && \forall p \in P, \label{c25}  \\
    &\sum_{p \in P} \sum_{r \in R_p : e \in r} z_r^p \le c_p y_e, && \forall e \in \mathcal{E}_{u_0}, \label{c26} \\
    &z_r^p \ge 0, && \forall p \in P, r \in R_p. \label{c27} 
\end{align}

Freight flows, however, involve intermediate transfers and unit-level handling, which leads to a large set of feasible paths. We therefore retain the arc-based formulation for freight to keep the model tractable.

Passenger demand is stochastic and presumed to follow a normal distribution, 
\(\tilde{q}_p \sim \mathcal{N}(\mu_p,\sigma_p^2)\). 
We adopt a chance-constrained approach to ensure that uncertain passenger demand can be served with a prescribed confidence level. 

A deterministic decision variable \(\widehat{q}_p\) is introduced to represent the planned passenger demand assigned to feasible paths for passenger group \(p\). 
This quantity represents the reliability-adjusted service requirement implied by the chance constraint rather than the realized passenger demand in each scenario. 

We then rewrite constraint~\eqref{c25} as
\begin{equation}
    \sum_{r \in R_p} z_r^p = \widehat{q}_p, \qquad \forall p \in P. \label{c28} 
\end{equation}

To ensure that this assigned amount covers stochastic demand with confidence level \(1-\alpha_p\), we impose
\begin{equation}
    \mathbb{P}\big(\widehat{q}_p \ge \tilde{q}_p\big) \ge 1-\alpha_p, 
    \qquad \forall p \in P.
\end{equation}

Given the normality of $\tilde{q}_p$, this chance constraint is equivalent to the following deterministic linear inequalities:
\begin{align}
    & \widehat{q}_p \ge \mu_p + \Phi^{-1}(1-\alpha_p)\,\sigma_p, && \forall p \in P, \label{c29} \\
    & \widehat{q}_p \ge 0, && \forall p \in P, \label{c30} 
\end{align}
where $\Phi^{-1}(\cdot)$ denotes the inverse cumulative distribution function of the standard normal distribution. The normality assumption provides a closed-form demand quantile for the chance constraint, while the formulation only requires the corresponding \(1-\alpha_p\) quantile. 
Therefore, other distributional assumptions or empirical quantiles can be incorporated without changing the remaining model structure. 

\section{Solving Methodology}
\label{Solving Methodology}


The integrated service--unit--flow formulation leads to a large-scale MILP, making direct solution computationally challenging.
We therefore develop two solution methods that exploit different model structures.

To solve exactly, we use a Benders decomposition based on the planning--scheduling coupling. 
The unit-scheduling part is a key source of fractional solutions in the relaxation, but becomes tractable once the planning decisions are fixed. 
We therefore solve it through subproblems and further decompose it by station, allowing multiple cuts to be generated in each iteration. 
Problem-specific valid inequalities and a warm-start strategy are added to improve convergence.

For larger instances, we use a different decomposition logic based on passenger priority. 
The two-stage heuristic first builds a passenger-oriented service plan and then incorporates freight flows and unit operations.

Before presenting the solution methods, we first generate the passenger path sets used in the path-based formulation. 
In the path-based formulation, passenger flows are assigned to complete paths, requiring a feasible path set \(R_p\) for each passenger group \(p \in P\). 
Because paths are restricted by travel direction, time windows, and the no-transfer requirement, the search space is limited. 
We enumerate all simple paths from the auxiliary origin \(o'_p\) to the auxiliary destination \(\delta'_p\) using depth-first search on the acyclic space--time--state network. 
Full enumeration preserves equivalence with the original arc-based model and is tractable for the instances considered. 
For larger instances, restricted path sets or column generation can be used to reduce model size, with the latter preserving exactness.



\subsection{Benders Decomposition}

The model jointly determines service planning, flow assignment, and unit scheduling on a space--time--state network, which yields a tightly coupled formulation. To improve tractability, we apply Benders decomposition that separates high-level service and flow decisions from low-level unit scheduling decisions. 

We partition the arc set $\mathcal E$ into master arcs $\mathcal E_{\mathrm{mas}}$, containing service and reconfiguration arcs, and subproblem arcs $\mathcal E_{\mathrm{sub}}$, containing storage arcs. 
The master problem determines departure decisions, vehicle routes, timetables, vehicle compositions, reconfiguration decisions, and passenger and freight demand flows. 
The storage-arc decisions and station-level initial fleet variables are handled in the subproblems. 

Given a master solution, the induced service and reconfiguration flows create net unit imbalances at storage-state nodes. 
The subproblem checks whether these imbalances admit feasible storage flows and computes the required initial fleet. 
Depending on the subproblem outcome, feasibility or optimality cuts are added to the master problem, progressively tightening the master relaxation until convergence. 

\subsubsection{Benders Subproblems}

The integrated model couples service planning and flow assignment decisions with detailed unit scheduling. We exploit this structure through Benders decomposition, which separates the problem into a high-level master problem and a collection of independent station-level subproblems.

Given a master solution $\hat{y}$, the subproblem decomposes into $|S|$ independent components, one for each station $s \in S$. This decomposition holds because unit scheduling decisions become station-wise decoupled once the inter-station flows implied by the service plan are fixed. The master problem uses auxiliary variables $\bm{\theta} = (\theta_s)_{s \in S}$ to approximate the fleet operating cost at each station.

Master decisions on $\mathcal E_{\mathrm{mas}}$ induce exogenous net inflows at nodes in the subnetwork. For any node $v=(s,t,k)$, we define the net inflow induced by a master flow vector $y$ as
\begin{equation}
b_v(y)
=
\sum_{e\in \mathcal E^{-}(v)\cap \mathcal E_{\mathrm{mas}}} y_e
-
\sum_{e\in \mathcal E^{+}(v)\cap \mathcal E_{\mathrm{mas}}} y_e .
\label{eq:bv_def}
\end{equation}
Thus, $b_v(y)$ can be positive or negative: positive values indicate units entering the storage subnetwork, while negative values indicate units required by the master arcs.

For each station $s \in S$, the subproblem $Q_s(\hat y)$ minimizes the local fleet size subject to flow conservation and capacity constraints at station $s$:
\begin{align}
Q_s(\hat y) = \min\quad 
& w_1 n_s \label{sp:obj_s}\\
\text{s.t.}\quad
& \sum_{e \in \mathcal E^{+}(s,\underline t,k_2)\cap\mathcal E_{\mathrm{sub}}} y_e - n_s = b_{(s,\underline t,k_2)}(\hat y), \label{sp:src_s}\\
& \sum_{e \in \mathcal E^{-}(s,\overline t,k_2)\cap\mathcal E_{\mathrm{sub}}} y_e - n_s = - b_{(s,\overline t,k_2)}(\hat y), \label{sp:sink_s}\\
& \sum_{e \in \mathcal E^{+}(s,t,k_2)\cap\mathcal E_{\mathrm{sub}}} y_e - \sum_{e \in \mathcal E^{-}(s,t,k_2)\cap\mathcal E_{\mathrm{sub}}} y_e = b_{(s,t,k_2)}(\hat y), && \forall t \in T \setminus \{\underline{t}, \overline{t}\}, \label{sp:bal_s}\\
& y_e \le \overline n_s, && \forall e \in \mathcal{E}_\circ(s), \label{sp:up_ye}\\
& n_s \le \overline n_s, \label{sp:up_ns}\\
& y_e \ge 0, && \forall e \in \mathcal{E}_\circ(s), \\
& n_s \ge 0.
\end{align}

Constraints~\eqref{sp:src_s}, \eqref{sp:sink_s}, and \eqref{sp:bal_s} are the station-level counterparts of Constraints~\eqref{c2}, \eqref{c4}, and \eqref{c5}, respectively. In the subproblem, we relax the integrality constraints on $y_e$ and $n_s$. 
Given an integer master solution $\hat y$, all induced node imbalances are integral. The resulting subproblem is a minimum-cost flow problem with a totally unimodular constraint matrix, so its linear relaxation admits an optimal integer solution whenever it is feasible. 

\subsubsection{Benders Cuts}

For each station $s$, we associate dual variables $\pi_s, \rho_s, \lambda_{(s,t)}, \xi_{(s,e)}$ and $\beta_s$ with constraints \eqref{sp:src_s}--\eqref{sp:up_ns}, respectively. We define the station-specific dual profit function as
\begin{equation}
\begin{aligned}
\Phi_s(y; \pi_s, \rho_s, \lambda_{(s,t)}, \xi_{(s,e)}, \beta_s)
:=\;& \pi_s b_{(s,\underline t,k_2)}(y)
- \rho_s b_{(s,\overline t,k_2)}(y) + \sum_{t \in T \setminus \{\underline{t}, \overline{t}\}} \lambda_{(s,t)} b_{(s,t,k_2)}(y) \\
&+ \overline n_s \sum_{e \in \mathcal E_\circ(s)} \xi_{(s,e)}
+ \overline n_s \beta_s .
\end{aligned}
\end{equation}

\paragraph{Optimality Cuts.} If subproblem $s$ is feasible, let $(\pi_s^{(i)}, \rho_s^{(i)}, \lambda_{(s,t)}^{(i)},\xi_{(s,e)}^{(i)}, \beta_s^{(i)})$ be an optimal dual solution at iteration $i$. We generate a cut:
\begin{equation}
\theta_s \ge \Phi_s(y; \pi_s^{(i)}, \rho_s^{(i)}, \lambda_{(s,t)}^{(i)},\xi_{(s,e)}^{(i)},\beta_s^{(i)}) \label{cut:opt}
\end{equation}

\paragraph{Feasibility Cuts.} If subproblem $s$ is infeasible, let $(\bar\pi_s^{(i)}, \bar\rho_s^{(i)}, \bar\lambda_{(s,t)}^{(i)}, \bar\xi_{(s,e)}^{(i)}, \bar\beta_s^{(i)})$ be a Farkas ray at iteration $i$. We generate a cut:
\begin{equation}
0 \ge \Phi_s(y; \bar\pi_s^{(i)}, \bar\rho_s^{(i)}, \bar\lambda_{(s,t)}^{(i)},\bar\xi_{(s,e)}^{(i)}, \bar\beta_s^{(i)}) \label{cut:fea}
\end{equation}

\subsubsection{Benders Master Problem}
Let $\mathcal{I}_s^{\mathrm{opt}}$ and $\mathcal{I}_s^{\mathrm{fea}}$ denote the sets of optimality and feasibility cuts generated for station $s$, respectively. 
At each iteration, we solve a restricted master problem (RMP) that includes the master decision variables and all Benders cuts accumulated up to the current iteration. 
Here, \(\theta_s\) represents the RMP approximation of the optimal fleet cost returned by the station-level unit-scheduling subproblem for station \(s\). 
The RMP is formulated as:
\begin{align}
\min\quad
& \sum_{s \in S} \theta_s + w_2 \sum_{p \in P} \sum_{r \in R_p} \tau_r z_r^p
+ w_3 \sum_{f \in F} \left(q_f - \sum_{e \in \mathcal{E}^{+}(o_f,a_f,k_2)} z_e^f\right) \\
\text{s.t.}\quad
& y_e \in \mathbb Z_{\ge 0},
&& \forall e\in\mathcal E_{\mathrm{mas}}, \label{mp:y-domain}\\
& \theta_s \ge 0, && \forall s \in S \\
& \theta_s \ge \Phi_s(y; \pi_s^{(i)}, \rho_s^{(i)}, \lambda_{(s,t)}^{(i)},\xi_{(s,e)}^{(i)},\beta_s^{(i)}), && \forall s \in S, i \in \mathcal{I}_s^{\mathrm{opt}} \\
& 0 \ge \Phi_s(y; \bar\pi_s^{(i)}, \bar\rho_s^{(i)}, \bar\lambda_{(s,t)}^{(i)},\bar\xi_{(s,e)}^{(i)}, \bar\beta_s^{(i)}), &&  \forall s \in S, i \in \mathcal{I}_s^{\mathrm{fea}} \\
& (\ref{c1}),(\ref{c6})-(\ref{c8}),(\ref{c10}),(\ref{c18})-(\ref{c23}),(\ref{c26})-(\ref{c28}),(\ref{c29})-(\ref{c30}).
\nonumber
\end{align}

Unlike the standard single-cut approach, which treats unit scheduling as one subproblem and generates at most one cut per iteration, the multi-cut formulation decomposes the subproblem by station and generates cuts in parallel.
This provides more localized approximations of the subproblem value function and improves convergence in practice.

\subsubsection{Valid Inequalities}
\label{sec:valid_ineq}
To strengthen the restricted master problem, we exploit a station-wise balance property implied by the closed and periodic unit-scheduling structure. 
\begin{proposition}
    Any master solution $\hat y$ that admits a feasible unit schedule must satisfy
\begin{equation}
\sum_{t \in T} b_{(s,t,k_2)}(\hat y) = 0, \quad \forall s \in S. \label{eq:station_balance}
\end{equation}
\end{proposition}

\begin{proof}
Fix a station $s \in S$ and consider any master solution $\hat y$ for which the station-level subproblem is feasible. Subtracting the sink-balance constraint \eqref{sp:sink_s} from the source-balance constraint \eqref{sp:src_s} and summing the flow-conservation constraints \eqref{sp:bal_s} over all intermediate time layers $t \in T \setminus \{\underline{t}, \overline{t}\}$ yields
\begin{equation}
\sum_{t \in T} \left( \sum_{e \in \mathcal{E}^{+}(s,t,k_2)\cap \mathcal{E}_{\mathrm{sub}}} y_e
- \sum_{e \in \mathcal{E}^{-}(s,t,k_2)\cap \mathcal{E}_{\mathrm{sub}}} y_e \right)
= \sum_{t \in T} b_{(s,t,k_2)}(\hat y).
\end{equation}
The left-hand side is a telescoping sum over the subproblem arcs. Because every subproblem arc is internal to station $s$, each such arc appears exactly once with a positive sign (as outgoing) and once with a negative sign (as incoming), so the sum equals zero. Therefore, \(\sum_{t \in T} b_{(s,t,k_2)}(\hat y)=0\).
\end{proof}
We add constraints \eqref{eq:station_balance} to the master problem to enforce station-wise balance of the net unit flow, thereby screening out infeasible master solutions.

\subsubsection{Warming Up}

Before solving the original integer master problem, we perform a warm-up phase in which we solve its linear relaxation to obtain a tight convex approximation of the feasible region. Candidate solutions from the relaxed master problem are evaluated by the subproblems, which iteratively generate feasibility and optimality Benders cuts. To prevent the master problem from becoming computationally burdened by an excessive number of constraints, we employ a dynamic cut-pruning mechanism: every ten iterations and upon termination, we remove all non-binding cuts.
The warm-up phase terminates when the linear relaxation is solved to optimality or when a time limit of 5 minutes is reached.

\subsection{A Tailored Two-Stage Decomposition Heuristic}

The Benders decomposition improves solvability, but the model still converges slowly for large-scale instances. 
To address this, we develop a tailored two-stage decomposition heuristic. 
Since passenger services are given higher priority and largely determine the basic service pattern, 
we first identify a passenger-oriented service backbone and then optimize freight operations and detailed unit circulation on top of it.
The key idea is not to fully separate passenger and freight planning, but to preserve the main passenger service structure while postponing the more flexible freight-related decisions.
This reduces the search space while keeping the heuristic close to the integrated problem.

\subsubsection{Stage 1: Passenger-Oriented Service Backbone}
In the first stage, we solve a reduced model that determines the passenger services to be operated and the corresponding passenger demand assignment based only on passenger demand.
The reduced model still accounts for fleet usage and unit-scheduling feasibility, preventing the passenger solution from requiring excessive or infeasible unit operations.
Thus, the Stage-1 solution provides a resource-aware passenger backbone rather than a purely demand-driven passenger timetable. 
The resulting model is: 
\begin{align}
\min \quad 
& w_1 \sum_{s \in S} n_s
+ w_2 \sum_{p \in P} \sum_{r \in R_p} \tau_r z_r^p \label{obj} \\
\text{s.t.} \quad
& y_e = 0,  &&\forall e \in \mathcal{E}_u,\; u \in U_f, \\
& (\ref{c1})-(\ref{c12}),\; (\ref{c26})-(\ref{c28}),\; (\ref{c29})-(\ref{c30}). \nonumber
\end{align}
Let the resulting passenger-related unit flows be denoted by $y_e = \overline{y}_e$ for all $e \in \mathcal{E}_{u_0}$. 
These flows define the passenger service pattern passed to the second stage. 

\subsubsection{Stage 2: Freight Planning and Unit Scheduling}

In the second stage, the passenger service pattern obtained in Stage~1 is preserved. On this basis, freight demand is introduced, and the remaining decisions are optimized jointly.
In particular, Stage~2 inserts freight services into the residual capacity left by the passenger solution while re-optimizing unit circulation and fleet allocation globally, so that both passenger and freight operations remain feasible under shared fleet limitations.
The resulting model is: 
\begin{align}
\min \quad 
& w_1 \sum_{s \in S} n_s
+ w_3 \sum_{f \in F} \left(q_f - \sum_{e \in \mathcal{E}^{+}(o_f,a_f,k_2)} z_e^f \right) \label{obj_stage2} \\
\text{s.t.}\quad
& y_e = \overline{y}_e, && \forall e \in \mathcal{E}_{u_0}, \\
& (\ref{c1})-(\ref{c12}),\; (\ref{c18})-(\ref{c23}).\nonumber
\end{align}

Since unmet freight demand is allowed with a penalty cost, the passenger solution obtained in Stage~1 always remains feasible in Stage~2.
The resulting Stage~2 solution therefore provides a complete feasible solution to the original integrated problem.
Passenger flow assignment is fixed in Stage~1, so the corresponding passenger waiting cost is constant in Stage~2 and omitted from the objective.
The final objective value combines the passenger waiting cost from Stage~1 with the fleet and freight costs from Stage~2.

\section{Experimental Design}
\label{sec:exper design}
This section describes the construction of the experimental instances and parameter settings used in the numerical study. These settings provide the basis for the computational analyses that follow.

\subsection{Instance Description}
We consider two corridor configurations derived from the public transport network of Gothenburg, the second largest city in Sweden, reflecting two typical operational settings: a hub-based corridor and an urban corridor.
Figure~\ref{fig:corridors} illustrates these two corridor settings used in the experiments.

\begin{figure}[ht]
    \centering
    \begin{subfigure}[b]{0.48\textwidth}
        \centering
        \includegraphics[width=0.85\textwidth]{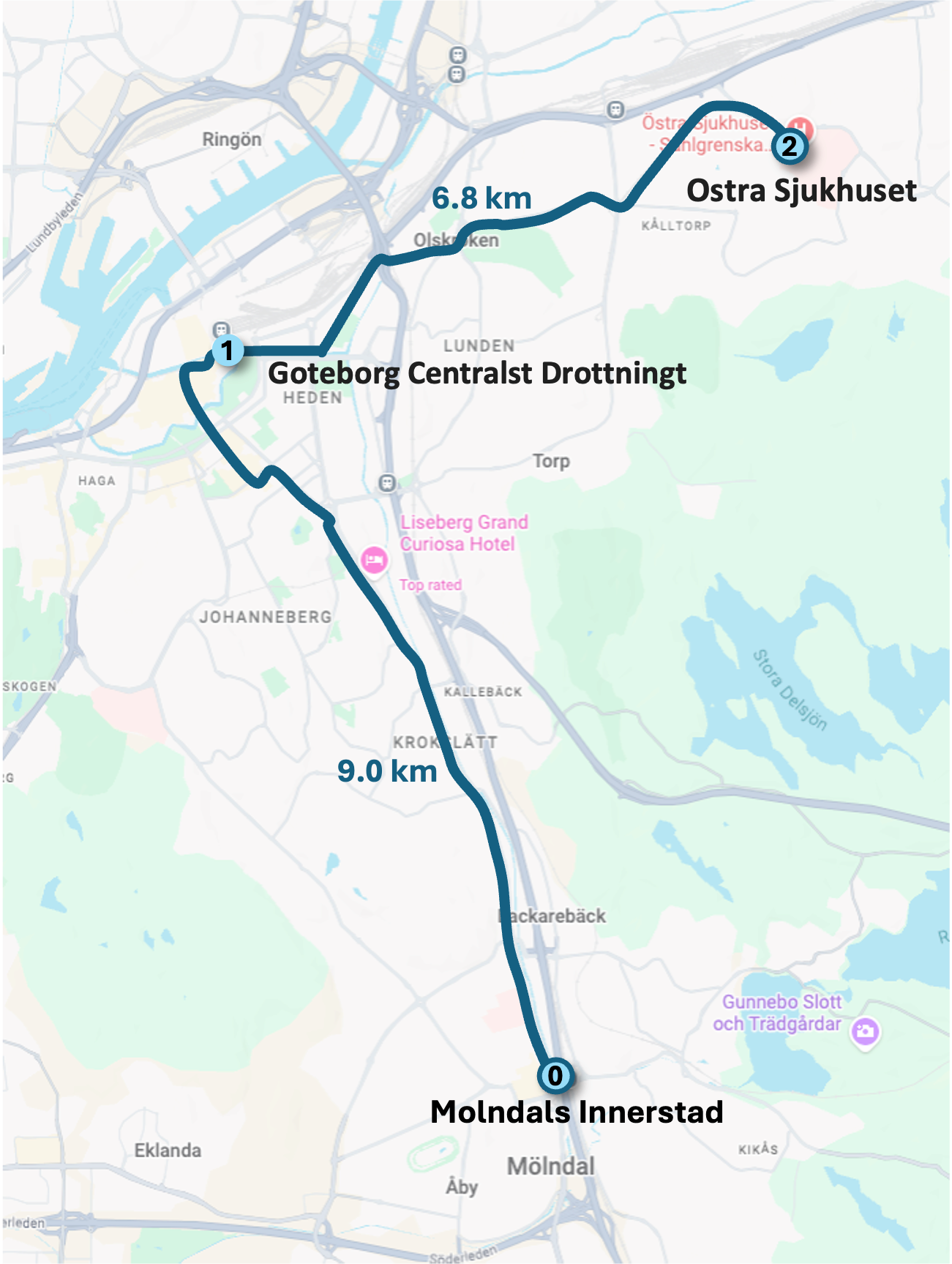}
        \caption{Hub-based corridor}
        \label{fig:hub}
    \end{subfigure}
    \hfill
    \begin{subfigure}[b]{0.48\textwidth}
        \centering
        \includegraphics[width=\textwidth]{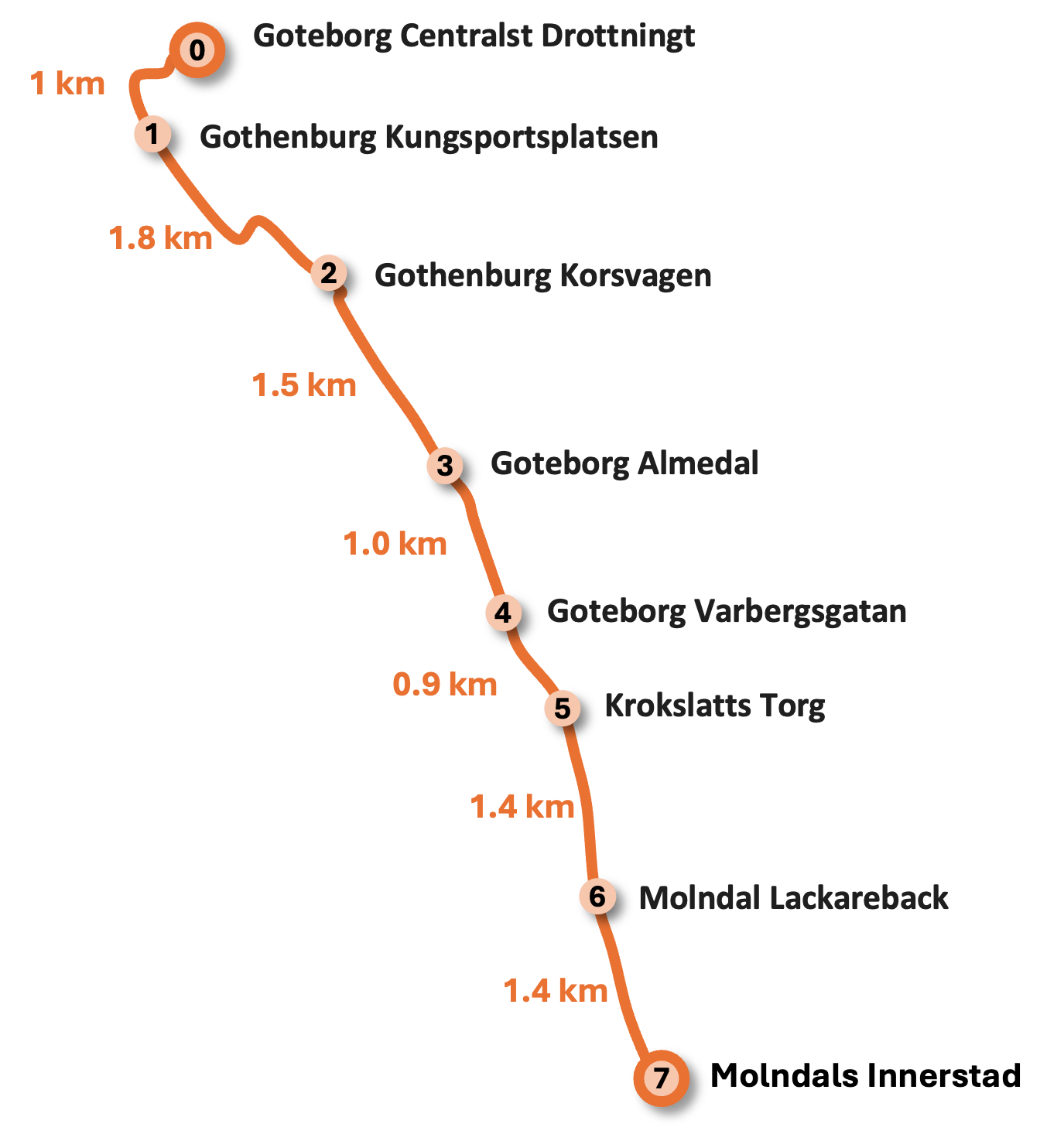}
        \caption{Urban corridor}
        \label{fig:urban}
    \end{subfigure}
    \caption{Corridor settings used in the experiments.}
    \label{fig:corridors}
\end{figure}

The hub-based corridor represents long-distance or hub-oriented transit services with limited intermediate stops.
It is modeled as a three-station corridor connecting Molndal, Goteborg Centralstation, and Ostra Sjukhuset. Molndal and Goteborg Centralstation serve as key terminal and interchange hubs, while Ostra Sjukhuset is a major destination node on the eastern side of the city.
Since the end-to-end travel time exceeds 30 minutes, we consider planning horizons of 1, 2, and 3 hours.

The urban corridor represents dense intra-city operations with frequent stops.
It is constructed by focusing on the Mölndal--Göteborg Centralstation segment and selecting six additional representative stations.
Because the end-to-end travel time is within 30 minutes, we examine shorter planning horizons of 0.5, 1, 1.5, 2, 2.5, and 3 hours.

Table~\ref{tab:instances} summarizes all test instances, including the corridor type, planning horizon, and the number of passenger and freight demand groups. For all test instances, the start of the planning horizon is set to zero, i.e., $\underline{t}=0$.

\begin{table}[ht]
\centering
\caption{Test instances and demand settings}
\label{tab:instances}
\begin{tabular}{ccccc}
\hline
Corridor & Instance & Horizon (min) & \#Passenger group & \#Freight group \\
\hline
\multirow{3}{*}{Hub-based} 
 & ins\_1  & [0, 60]  & 30  & 6 \\
 & ins\_2 & [0, 120] & 100 & 20 \\
 & ins\_3 & [0, 180] & 150 & 50 \\
\hline
\multirow{6}{*}{Urban} 
 & ins\_4  & [0, 30]  & 30  & 10 \\
 & ins\_5  & [0, 60]  & 70  & 15 \\
 & ins\_6  & [0, 90]  & 140 & 30 \\
 & ins\_7 & [0, 120] & 180 & 45 \\
 & ins\_8 & [0, 150] & 220 & 60 \\
 & ins\_9 & [0, 180] & 280 & 75 \\
\hline
\end{tabular}
\end{table}

Within each instance, passenger and freight demand groups are designed to be consistent with the baseline timetable under the prescribed headway settings and estimated segment loads.
Each group is characterized by an origin, a destination, a demand quantity, a service time window, and a travel direction. 
Passenger groups have specific arrival times, uncertain quantities, and relatively tight arrival time windows, reflecting the stochastic and time-sensitive nature of passenger demand.
Freight groups have deterministic quantities and wider service time windows, reflecting their greater scheduling flexibility.
Table~\ref{tab:sample_demand} provides illustrative examples of the passenger and freight demands for instance ins\_1.

\begin{table}[ht]
\centering
\caption{Illustrative passenger and freight demands for instance ins\_1}
\label{tab:sample_demand}
\begin{tabular}{ccccccccc}
\hline
Type & Group & $o$ & $d$ & Quantity & Arrival time & Left TW & Right TW & State \\
\hline
Passenger & 1 & 0 & 1 & $\mathcal{N}(6,1^2)$  & 0  & 6  & 36 & up   \\
Passenger & 2 & 1 & 2 & $\mathcal{N}(9,1^2)$  & 21 & 23 & 53 & up   \\
Passenger & 3 & 0 & 2 & $\mathcal{N}(7,1^2)$  & 10 & 33 & 60 & up   \\
Passenger & 4 & 2 & 1 & $\mathcal{N}(11,2^2)$ & 16 & 22 & 52 & down \\
Passenger & 5 & 1 & 0 & $\mathcal{N}(10,2^2)$ & 17 & 23 & 53 & down \\
\hline
Freight   & 1 & 0 & 1 & 20 & -- & 0  & 30 & up   \\
Freight   & 2 & 0 & 2 & 39 & -- & 0  & 60 & up   \\
Freight   & 3 & 1 & 0 & 4  & -- & 30 & 60 & down \\
Freight   & 4 & 2 & 1 & 13 & -- & 0  & 30 & down \\
Freight   & 5 & 2 & 0 & 17 & -- & 0  & 60 & down \\
\hline
\end{tabular}
\end{table}

In the subsequent computational experiments, both deterministic and stochastic passenger-demand settings are considered.
For a passenger group $p$ whose demand follows $\mathcal{N}(\mu_p,\sigma_p^2)$, the deterministic setting uses $\mu_p$ as the group demand, whereas the stochastic setting uses the demand quantile implied by the chance constraint in Section~\ref{Model reFormulation}.

\subsection{Parameter Settings}

The operating speed of modular vehicles is set to 0.5~km/min (30~km/h), which is consistent with the speeds commonly used for urban bus priority lanes and dedicated-corridor services. The passenger time coefficient is set to 3~kr/min to reflect waiting and delay costs. This value is broadly consistent with Swedish appraisal practice, where delays in public transport are valued more highly than normal in-vehicle time. The operating cost per unit is set to 500~kr/h as an aggregate coefficient for urban bus operations.

The remaining parameters are treated as scenario settings that reflect the assumed modular-vehicle design and corridor operating conditions. Each unit can carry up to 15 passengers or 30 boxes of freight. The maximum vehicle length is limited to five units. Storage capacity at intermediate stations is set to five units, and no storage limit is imposed at terminal stations. Due to differences in stop frequency and operating conditions, the maximum number of freight-carrying units per vehicle is set to two in the hub-based corridor and one in the urban corridor. The docking and undocking time for each unit is set to 1~min. The minimum headway between two consecutive departures from the same station is set to 8~min.





\section{Computational Performance}
\label{Computational Performance}

All algorithms are implemented in Java and solved with IBM ILOG CPLEX 22.1.1. The experiments are conducted on a MacBook Pro equipped with an Apple M4 Pro chip and 24~GB of memory. Each result is based on five independent runs to reduce the impact of randomness in the solution process. The time limit for each run is 30~min.

The arc-based formulation is solved directly by CPLEX and used as a benchmark. Benders decomposition is implemented using a generic callback that generates and adds cuts at candidate solutions during the branch-and-bound process.

For the decomposition heuristic, we use different time limits for the two stages. The Stage 1 model is solved with a time limit of 25~min, and the Stage 2 model is limited to 5~min. This setting is based on preliminary tests, which indicate that the Stage 2 problem can be solved efficiently once the passenger-related structure is fixed.

We evaluate the proposed methods on the two corridor settings described in Section~\ref{sec:exper design} under both deterministic (\textit{det}) and stochastic (\textit{sto}) passenger demand scenarios. 
For the stochastic setting, the chance-constraint confidence level is set to $1-\alpha_p=0.9$. 
Tables~\ref{tab:small_results} and~\ref{tab:large_results} summarize the results for the hub-based and urban corridor instances, respectively. 
We compare three approaches: the direct arc-based formulation solved by CPLEX (\textit{Solver}), the proposed Benders decomposition (\textit{Benders}), and the two-stage decomposition heuristic (\textit{Heuristic}).
The reported metrics include computation time ($t$, in seconds), objective value (Obj, in kr), the optimality gap when applicable ($g_{\text{opt}}$, in \%), and the relative gap to the solver incumbent ($g_{\text{rel}}$, in \%). 
A negative $g_{\text{rel}}$ indicates that Benders or the heuristic obtained a better solution than the direct solver within the same time limit.

\begin{table}[ht]
\centering
\renewcommand{\arraystretch}{1.2}
\caption{Computational results on hub-based corridor instances}
\label{tab:small_results}
\resizebox{\textwidth}{!}{
\begin{tabular}{ll
r r r
r r r r
r r r}
\hline
 &  & \multicolumn{3}{c}{Solver} & \multicolumn{4}{c}{Benders} & \multicolumn{3}{c}{Heuristic} \\
\cline{3-5} \cline{6-9} \cline{10-12}
Inst & Sce
& \begin{tabular}[c]{@{}c@{}}$t$\\(s)\end{tabular} 
& \begin{tabular}[c]{@{}c@{}}Obj\\(kr)\end{tabular} 
& \begin{tabular}[c]{@{}c@{}}$g_{\text{opt}}$\\(\%)\end{tabular}
& \begin{tabular}[c]{@{}c@{}}$t$\\(s)\end{tabular} 
& \begin{tabular}[c]{@{}c@{}}Obj\\(kr)\end{tabular} 
& \begin{tabular}[c]{@{}c@{}}$g_{\text{opt}}$\\(\%)\end{tabular} 
& \begin{tabular}[c]{@{}c@{}}$g_{\text{rel}}$\\(\%)\end{tabular}
& \begin{tabular}[c]{@{}c@{}}$t$\\(s)\end{tabular} 
& \begin{tabular}[c]{@{}c@{}}Obj\\(kr)\end{tabular} 
& \begin{tabular}[c]{@{}c@{}}$g_{\text{rel}}$\\(\%)\end{tabular} \\
\hline

\multirow{2}{*}{ins\_1}
& det & 4    & \textbf{9.37e3} & 0.00 
& 2    & \textbf{9.37e3} & 0.00 & 0.00 
& 1    & 1.04e4 & 10.67 \\
& sto & 5    & \textbf{9.80e3} & 0.00 
& 3    & \textbf{9.80e3} & 0.00 & 0.00 
& 2    & 1.39e4 & 30.12 \\
\hline

\multirow{2}{*}{ins\_2}
& det & 1800 & 2.90e4 & 9.73 
& 1800 & \textbf{2.77e4} & 4.75 & -4.37 
& 23   & 2.90e4 & 0.07 \\
& sto & 1800 & 3.37e4 & 10.53 
& 1800 & \textbf{3.23e4} & 6.33 & -4.01 
& 61   & 3.35e4 & -0.69 \\
\hline

\multirow{2}{*}{ins\_3}
& det & 1800 & 5.03e4 & 19.32 
& 1800 & 4.75e4 & 13.78 & -5.59 
& 1501 & \textbf{4.69e4} & -6.69 \\
& sto & 1800 & 5.91e4 & 22.45 
& 1800 & 5.52e4 & 15.93 & -6.63 
& 1501 & \textbf{5.39e4} & -8.78 \\
\hline

\end{tabular}
}
\end{table}

\begin{table}[ht]
\centering
\renewcommand{\arraystretch}{1.2}
\caption{Computational results on urban corridor instances}
\label{tab:large_results}
\resizebox{\textwidth}{!}{
\begin{tabular}{llrrr rrrr rrr}
\hline
 &  & \multicolumn{3}{c}{Solver} & \multicolumn{4}{c}{Benders} & \multicolumn{3}{c}{Heuristic} \\
\cline{3-5} \cline{6-9} \cline{10-12}
Inst & Sce
& \begin{tabular}[c]{@{}c@{}}$t$\\(s)\end{tabular} 
& \begin{tabular}[c]{@{}c@{}}Obj\\(kr)\end{tabular} 
& \begin{tabular}[c]{@{}c@{}}$g_{\text{opt}}$\\(\%)\end{tabular}
& \begin{tabular}[c]{@{}c@{}}$t$\\(s)\end{tabular} 
& \begin{tabular}[c]{@{}c@{}}Obj\\(kr)\end{tabular} 
& \begin{tabular}[c]{@{}c@{}}$g_{\text{opt}}$\\(\%)\end{tabular} 
& \begin{tabular}[c]{@{}c@{}}$g_{\text{rel}}$\\(\%)\end{tabular}
& \begin{tabular}[c]{@{}c@{}}$t$\\(s)\end{tabular} 
& \begin{tabular}[c]{@{}c@{}}Obj\\(kr)\end{tabular} 
& \begin{tabular}[c]{@{}c@{}}$g_{\text{rel}}$\\(\%)\end{tabular} \\
\hline

\multirow{2}{*}{ins\_4}
& det & 15 & \textbf{5.13e3} & 0.00 
& 3 & \textbf{5.13e3} & 0.00 & 0.00 
& 1 & 5.90e3 & 14.98 \\
& sto & 11 & \textbf{5.73e3} & 0.00 
& 6 & \textbf{5.73e3} & 0.00 & 0.00 
& 1 & 6.46e3 & 12.69 \\
\hline

\multirow{2}{*}{ins\_5}
& det & 1800 & 1.49e4 & 3.04 
& 1800 & \textbf{1.49e4} & 1.10 & -0.50 
& 22 & 1.76e4 & 17.89 \\
& sto & 1800 & 1.81e4 & 7.01 
& 1800 & \textbf{1.76e4} & 1.63 & -2.49 
& 90 & 2.15e4 & 18.71 \\
\hline

\multirow{2}{*}{ins\_6}
& det & 1800 & 2.86e4 & 12.63 
& 1800 & \textbf{2.68e4} & 5.76 & -6.44 
& 613 & 2.95e4 & 3.00 \\
& sto & 1800 & 3.44e4 & 17.19 
& 1800 & \textbf{3.21e4} & 9.24 & -6.84 
& 1517 & 3.47e4 & 0.70 \\
\hline

\multirow{2}{*}{ins\_7}
& det & 1800 & 4.30e4 & 19.86 
& 1800 & 4.06e4 & 11.71 & -5.78 
& 1553 & \textbf{4.02e4} & -6.57 \\
& sto & 1800 & 5.48e4 & 28.51 
& 1800 & 5.04e4 & 17.88 & -7.90 
& 1506 & \textbf{4.76e4} & -13.13 \\
\hline

\multirow{2}{*}{ins\_8}
& det & 1800 & 5.71e4 & 27.93 
& 1800 & 5.03e4 & 15.47 & -11.85 
& 1604 & \textbf{4.88e4} & -14.52 \\
& sto & 1800 & 1.34e5 & 63.83 
& 1800 & 6.84e4 & 24.44 & -49.08 
& 1511 & \textbf{6.18e4} & -53.98 \\
\hline

\multirow{2}{*}{ins\_9}
& det & 1800 & 1.78e5 & 73.46 
& 1800 & 7.40e4 & 30.52 & -58.50 
& 1800 & \textbf{6.15e4} & -65.51 \\
& sto & 1800 & 2.68e5 & 79.12 
& 1800 & 8.69e4 & 31.88 & -67.54 
& 1681 & \textbf{7.24e4} & -72.95 \\
\hline

\end{tabular}
}
\end{table}

For the hub-based instances (Table~\ref{tab:small_results}), both exact approaches solve the smallest instance ins\_1 to optimality, with Benders being faster.
As the instance size increases, both exact methods hit the time limit; however, Benders consistently returns better objective values and smaller optimality gaps than the solver.

Convergence in exact approaches becomes difficult for larger cases due to highly fractional and large relaxation gaps induced by the integrated network-flow structure.
In this regime, the heuristic becomes more competitive.
For ins\_3, it achieves the best objective value among all methods in both demand scenarios, while requiring shorter solution times. A similar pattern is observed for the urban corridor instances (Table~\ref{tab:large_results}), but the problem is substantially more challenging.
Compared with the hub-based setting, the denser station layout yields a larger network with more arcs and OD pairs.

For small and medium-sized instances, Benders performs better, as reflected by shorter solution times and smaller optimality gaps.
For example, in ins\_6, the solver reports gaps of 12.63\% and 17.19\%, whereas Benders reduces them to 5.76\% and 9.24\%, respectively.
Due to its greedy construction, the heuristic is less competitive in this regime.
Nevertheless, it can still generate feasible solutions quickly. 
For ins\_6, the heuristic already produces objective values close to those obtained by the solver.

For larger instances, the performance of the exact methods deteriorates markedly.
For ins\_7-9, the solver's optimality gaps rise to as much as 79.12\%, whereas Benders reduces the worst-case gap to 31.88\%. However, even Benders is unable to close the gaps within the time limit. 
In contrast, the heuristic consistently delivers the best objective values in this group of instances. For Ins\_9, the largest instance tested, both the solver and Benders exhibit large optimality gaps under the deterministic and stochastic settings. 
Despite the lack of optimality guarantees, the heuristic delivers substantially higher-quality feasible solutions, improving the solver solution by 65.51\% with the same computation time in the deterministic setting and by 72.95\% with less computation time in the stochastic setting.

Comparing deterministic and stochastic scenarios, the stochastic setting generally yields higher objective values due to additional service reliability requirements.
However, the relative performance of the three methods remains consistent across the two scenarios.

Overall, the results on the urban corridor further confirm that Benders improves exact solution quality and performs well on small and medium-sized instances, while the heuristic becomes increasingly advantageous as the problem scale and complexity grow.

\section{Sensitivity Analysis}
\label{Sensitivity Analysis}

This section examines how passenger demand uncertainty and passenger--freight temporal overlap affect system performance and resource allocation. 
For each scenario, the reported results are based on the best solution obtained from five independent runs. 

\subsection{Effect of Demand Uncertainty}

This subsection examines the impact of passenger demand uncertainty on system performance.
Using the chance-constrained model, we vary the confidence level to quantify how increasingly conservative decisions affect system cost and fleet size.
Out-of-sample performance is evaluated using 50 randomly generated demand scenarios.

Table~\ref{tab:uncertainty} summarizes the objective value (Obj, in kr) and the number of units (\#U) for the deterministic baseline and the stochastic solutions under different confidence levels. 
For each instance, the objective value increases as the confidence level increases.
Operationally, higher confidence levels require more modular units because additional capacity is needed to ensure passenger-service reliability.
This effect is more pronounced for larger instances at high confidence levels.
In particular, for ins\_6--ins\_9, the fleet size increases substantially when the confidence level rises from 0.9 to 0.99.

\begin{table}[ht]
\centering
\renewcommand{\arraystretch}{1.2}
\caption{Impact of demand uncertainty under different confidence levels}
\label{tab:uncertainty}
\resizebox{\textwidth}{!}{
\begin{tabular}{ccc ccc ccc ccc ccc}
\hline

\multirow{4}{*}{Inst}
& \multicolumn{2}{c}{\multirow{2}{*}{Deterministic}}
& \multicolumn{10}{c}{Confidence level $(1-\alpha_p)$} \\
\cline{4-13}

& \multicolumn{2}{c}{}
& \multicolumn{2}{c}{$0.6$}
& \multicolumn{2}{c}{$0.7$}
& \multicolumn{2}{c}{$0.8$}
& \multicolumn{2}{c}{$0.9$}
& \multicolumn{2}{c}{$0.99$} \\
\cline{2-13}

& Obj & \multirow{2}{*}{\#U}
& Obj & \multirow{2}{*}{\#U}
& Obj & \multirow{2}{*}{\#U}
& Obj & \multirow{2}{*}{\#U}
& Obj & \multirow{2}{*}{\#U}
& Obj & \multirow{2}{*}{\#U} \\

& (kr) &
& (kr) &
& (kr) &
& (kr) &
& (kr) &
& (kr) & \\

\hline

ins\_1 & 9.37e3 & 16 & 9.48e3 & 16 & 9.65e3 & 16 & 9.71e3 & 16 & 9.80e3 & 16 & 1.10e4 & 17 \\
ins\_2 & 2.76e4 & 22 & 2.92e4 & 23 & 3.00e4 & 24 & 3.11e4 & 24 & 3.22e4 & 25 & 3.53e4 & 27 \\
ins\_3 & 4.69e4 & 25 & 5.09e4 & 25 & 5.23e4 & 26 & 5.36e4 & 26 & 5.38e4 & 27 & 6.06e4 & 29 \\
ins\_4 & 5.13e3 & 12 & 5.20e3 & 13 & 5.49e3 & 13 & 5.61e3 & 14 & 5.73e3 & 14 & 6.05e3 & 14 \\
ins\_5 & 1.49e4 & 17 & 1.58e4 & 19 & 1.64e4 & 19 & 1.74e4 & 20 & 1.76e4 & 20 & 1.99e4 & 21 \\
ins\_6 & 2.67e4 & 20 & 2.92e4 & 20 & 2.98e4 & 21 & 3.08e4 & 22 & 3.20e4 & 22 & 4.10e4 & 25 \\
ins\_7 & 3.99e4 & 20 & 4.58e4 & 23 & 4.70e4 & 25 & 4.71e4 & 26 & 4.74e4 & 26 & 6.48e4 & 31 \\
ins\_8 & 4.85e4 & 23 & 5.21e4 & 24 & 5.76e4 & 25 & 5.79e4 & 26 & 6.18e4 & 26 & 7.01e4 & 31 \\
ins\_9 & 6.14e4 & 23 & 6.50e4 & 26 & 6.81e4 & 27 & 7.18e4 & 27 & 7.19e4 & 27 & 9.02e4 & 32 \\

\hline
\end{tabular}
}
\end{table}

To further evaluate these solutions, we conduct a Monte Carlo simulation with 50 randomly generated passenger demand scenarios.
Because freight demand and operational schedules are fixed, only passenger demand varies across scenarios.
For each scenario, we solve a relaxed passenger assignment model under the given service plan, in which unmet demand is allowed but heavily penalized (see Appendix~\ref{apxA}).
The evaluation focuses on average passenger waiting time and the number of unserved passengers. 

\begin{figure}[h]
\centering
\includegraphics[width=\linewidth]{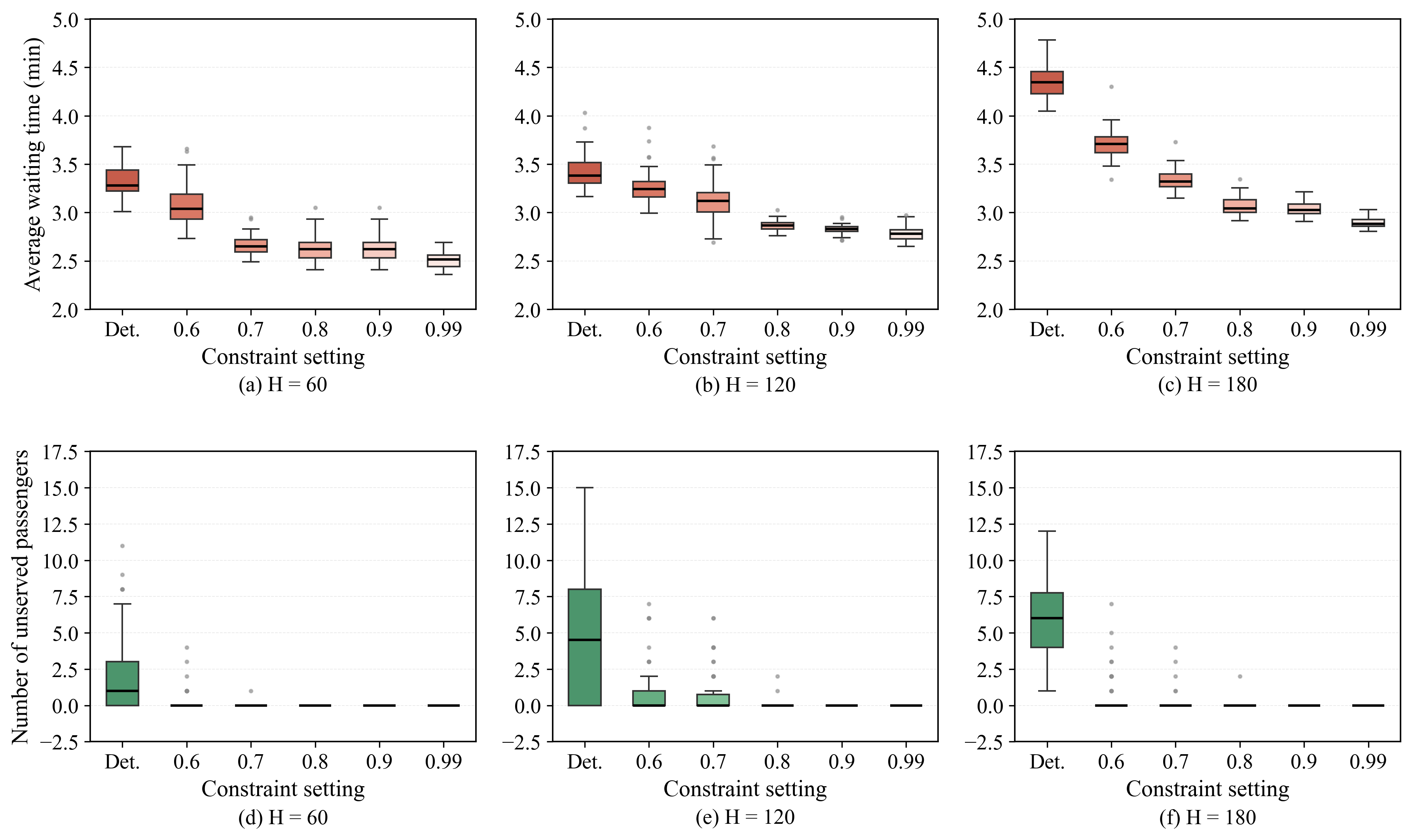}
\caption{Service performance under different confidence levels for hub-based instances.}
\label{fig:box3}
\end{figure}

\begin{figure}[H]
\centering
\includegraphics[width=\linewidth]{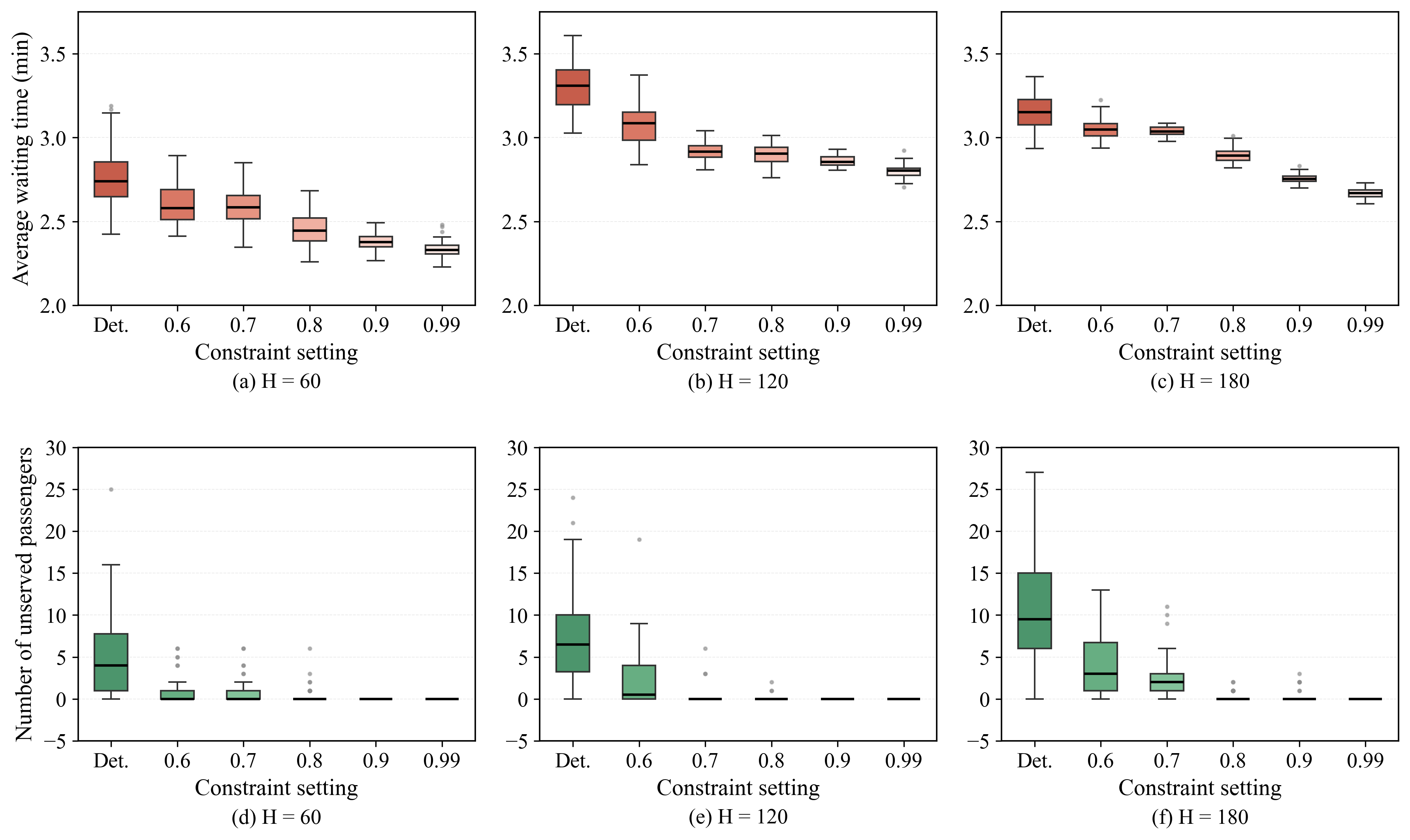}
\caption{Service performance  under different confidence levels for urban corridor instances.}
\label{fig:box8}
\end{figure}

Figures~\ref{fig:box3} and~\ref{fig:box8} present results for representative instances across different confidence levels and planning horizons.
The deterministic solution performs poorly under demand variability, leading to longer passenger waiting times and more unserved passengers.
This highlights the importance of accounting for demand uncertainty at the planning stage.
Introducing chance constraints improves service performance.
Both the average waiting time and the number of unserved passengers decrease as the confidence level increases.
However, the marginal improvement diminishes at higher confidence levels.
In particular, increasing the confidence level from 0.9 to 0.99 yields only small additional reductions in waiting time and little change in service level relative to the 0.8--0.9 cases.

Overall, the results reveal a clear trade-off between robustness and efficiency.
In a passenger-priority system, incorporating demand uncertainty is necessary to ensure service quality.
Nevertheless, very high confidence levels require substantially more capacity while providing limited additional benefits.
In our experiments, confidence levels around 0.8--0.9 offer a balanced degree of conservatism.

\subsection{Effects of Passenger--Freight Temporal Overlap}

This section examines how different degrees of temporal overlap between passenger and freight demand affect system performance.
Passenger-related decisions are obtained at a fixed confidence level of 0.9, while the time windows of freight requests are gradually relaxed from tight intervals to the full planning horizon.
This design allows us to control the overlap between the two demand types and assess its impact on system performance.

To quantify temporal overlap, we use a normalized index based on the relaxation of freight time windows.
All freight requests initially share the same time window $[0,30]$, and the right endpoint is extended in increments of 30~min until it reaches the full planning horizon.
We define the temporal-overlap index as
\begin{equation}
\phi = \frac{H-\bar{b}_f}{H-30},
\end{equation}
where $\bar{b}_f$ denotes the right endpoint of freight time windows.
A larger value of $\phi$ indicates a tighter freight time window and less flexibility to coordinate with passenger demand.

\begin{table}[t]
\centering
\caption{Impact of temporal overlap on system performance}
\label{tab:overlap}
\begin{tabular}{c c cc cc}
\hline
\multirow{2}{*}{Horizon} & \multirow{2}{*}{$\phi$} 
& \multicolumn{2}{c}{Hub-based} 
& \multicolumn{2}{c}{Urban} \\
\cline{3-6}
 &  & Obj (kr) & Freight served (\%) & Obj (kr)  & Freight served (\%) \\

\hline

\multirow{2}{*}{[0, 60]} 
& 0.00   & 9.27e3 & 100.00 & 1.61e4 & 100.00 \\
& 1.00   & 1.83e4 & 50.44  & 1.88e4 & 100.00 \\

\hline

\multirow{4}{*}{[0, 120]} 
& 0.00    & 3.05e4 & 100.00 & 4.39e4 & 100.00 \\
& 0.33 & 3.67e4 & 98.46  & 4.48e4 & 100.00 \\
& 0.67 & 5.77e4 & 67.69  & 5.71e4 & 87.92  \\
& 1.00    & 8.19e4 & 28.97  & 7.69e4 & 61.95  \\

\hline

\multirow{6}{*}{[0, 180]} 
& 0.00   & 4.77e4 & 100.00 & 6.69e4 & 98.95  \\
& 0.20 & 4.92e4 & 100.00 & 6.89e4 & 98.20  \\
& 0.40 & 5.34e4 & 100.00 & 7.30e4 & 97.75  \\
& 0.60 & 7.60e4 & 86.03  & 8.73e4 & 87.09  \\
& 0.80 & 1.14e5 & 51.90  & 1.05e5 & 73.12  \\
& 1.00   & 1.41e5 & 20.95  & 1.41e5 & 46.25  \\

\hline
\end{tabular}
\end{table}

Table~\ref{tab:overlap} summarizes the impact of temporal overlap on the objective value (Obj, in kr) and freight service rate (Freight served, in \%) under different planning horizons.
As $\phi$ increases, freight service rates consistently decrease in both corridor settings, indicating stronger competition when passenger and freight demands are more temporally aligned.
Operationally, this pattern reflects the passenger-first priority of the system.
When temporal overlap is high, the space for unit circulation is limited. Therefore, capacity is primarily allocated to passenger services, which reduces the freight service level.
As freight time windows are relaxed (i.e., lower $\phi$), freight requests gain scheduling flexibility, improving the freight service rate without degrading passenger performance.

Overall, temporal overlap is a key driver of passenger--freight interaction.
More flexible freight scheduling reduces overlap, eases competition, and improves system efficiency.

\section{Comparison with Benchmark Systems}
\label{Comparison with Systems}
This section compares the proposed Modular Integrated System (MIS) with two benchmark systems using representative instances from both corridor settings.
The comparison evaluates the benefits of capacity sharing and dynamic reconfiguration for serving passenger and freight demands.
It also demonstrates how alternative operational settings can be represented through modifications to the network structure and model constraints.

\begin{table}[t]
\centering
\caption{Comparison of MIS, MSS, and FCTS under different scenarios}
\label{tab:system_comparison}
\begin{tabular}{c c c c c c c c c}
\hline
\multirow{2}{*}{Corridor} 
& \multirow{2}{*}{Horizon} 
& \multirow{2}{*}{System} 
& Obj 
& \multirow{2}{*}{\#U} 
& \multirow{2}{*}{\#UP} 
& \multirow{2}{*}{\#UF} 
& APW 
& USF \\
& 
& 
& (kr) 
& 
& 
& 
& (min) 
& (box) \\
\hline

\multirow{9}{*}{Hub-based}
& \multirow{3}{*}{[0, 60]}
& MIS  & 9.80e3 & 16 & -  & -  & 2.53 & 0 \\
&      & MSS  & 1.10e4 & 17  & 10 & 7  & 2.74 & 0 \\
&      & FCTS & 1.87e4 & 26  & 20 & 6  & 2.57 & 28 \\
\cline{2-9}

& \multirow{3}{*}{[0, 120]}
& MIS  & 3.22e4 & 25 & -  & -  & 3.10 & 0 \\
&      & MSS  & 3.39e4 & 27  & 17 & 10 & 2.60 & 4 \\
&      & FCTS & 4.42e4 & 29  & 20 & 9  & 4.77 & 4 \\
\cline{2-9}

& \multirow{3}{*}{[0, 180]}
& MIS  & 5.38e4 & 27 & -  & -  & 2.49 & 20 \\
&      & MSS  & 7.59e4 & 31  & 19  & 12  & 2.69 & 55 \\
&      & FCTS & 7.14e4 & 32  & 20 & 12 & 5.16 & 2 \\
\hline

\multirow{9}{*}{Urban}
& \multirow{3}{*}{[0, 60]}
& MIS  & 1.76e4 & 20 & -  & -  & 2.47 & 0 \\
&      & MSS  & 1.96e4 & 24  & 18 & 6  & 2.44 & 2 \\
&      & FCTS & 2.36e4 & 29  & 20 & 9  & 2.98 & 0 \\
\cline{2-9}

& \multirow{3}{*}{[0, 120]}
& MIS  & 4.74e4 & 26 & -  & -  & 2.67 & 2 \\
&      & MSS  & 4.96e4 & 27  & 20 & 7  & 2.52 & 28 \\
&      & FCTS & 5.02e4 & 26  & 20 & 6  & 2.87 & 0 \\
\cline{2-9}

& \multirow{3}{*}{[0, 180]}
& MIS  & 7.19e4 & 27 & -  & -  & 2.57 & 28 \\
&      & MSS  & 7.69e4 & 30  & 21 & 9  & 2.54 & 46 \\
&      & FCTS & 7.36e4 & 26  & 20 & 6  & 2.81 & 10 \\
\hline
\end{tabular}
\end{table}

The first benchmark is a Modular Segregated System (MSS), which retains the modular structure but does not allow capacity sharing between passenger and freight services.
Each modular unit is pre-assigned to either passenger or freight transport and cannot switch between the two.
The second benchmark is the Fixed-Composition Transit System (FCTS), which represents conventional transit operations.
Vehicles operate with fixed compositions along the entire route, and dynamic reconfiguration is not allowed.
Both benchmark systems are represented within the same modeling framework through targeted modifications. 
To ensure a consistent comparison, all systems are evaluated under the same demand settings and parameter values.

\input{figures/11_different_system_s3}

For MSS, additional storage arcs are introduced to distinguish freight-unit flows, and the unit-flow balance and unit-allocation constraints are adjusted to enforce the separation between passenger and freight units (Appendix~\ref{apxB}).

For FCTS, all reconfiguration-related arcs are disabled and reconfiguration at intermediate stations is prohibited. 
Passenger and freight services are optimized separately under fixed vehicle compositions, with five units for passenger vehicles and three for freight vehicles. 
The fixed compositions are selected to match the maximum passenger vehicle length allowed in the modular system while providing dedicated freight vehicles with substantial carrying capacity.
Alternative fixed-composition settings can also be evaluated within the same benchmark framework.
Following common practice in the literature, freight loading and unloading are assumed to occur during station dwell times, so docking and undocking times are set to zero. 
Additional constraints are imposed to ensure a consistent unit composition along service arcs (Appendix~\ref{apxC}). 

\input{figures/12_different_system_s8}

Table~\ref{tab:system_comparison} reports the results across corridor types and planning horizons.
\#U denotes the total number of modular units used, whereas \#UP and \#UF denote the numbers of passenger and freight units used in MSS and FCTS, respectively.
APW denotes the average passenger waiting time, and USF denotes unserved freight demand.

Across all instances, MIS attains lower objective values than other benchmark systems. The improvement is mainly driven by better fleet utilization and better service outcomes.

From the resource perspective, MIS uses fewer units than the combined passenger and freight fleets required by MSS and FCTS in almost all instances. This result highlights the benefit of capacity pooling, since units can be reassigned between passenger and freight functions over time and the system can therefore meet demand with a smaller fleet.

In terms of service performance, MIS keeps passenger waiting times stable while limiting unserved freight. In MSS, passenger and freight services share the same fleet and compete for capacity. Because passenger demand is prioritized, freight service is more likely to be reduced when conflicts occur. FCTS exhibits the opposite trade-off. By operating passenger and freight services as two independent systems, it avoids direct competition. 
However, passengers at intermediate stations cannot be served by vehicles starting from their station or a closer upstream station and must instead wait for vehicles departing from terminals, which significantly increases their waiting time.
In addition, separate planning and control of passenger and freight services may lead to fragmented decisions and weaker coordination.

Figures~\ref{fig:Comparative analysis of hub schemes} and~\ref{fig:Comparative analysis of urban schemes} visualize the solutions for the hub-based and urban corridors with a 60-minute planning horizon. In each figure, the left panel shows the timetable and vehicle compositions along service arcs. For MIS and MSS, solid lines denote aggregated passenger services, and the number at the origin of each outgoing passenger arc indicates the number of passenger units. Each dashed line denotes a freight service operated by a single freight unit. For FCTS, passenger and freight timetables are overlaid in the same panel although they are operated as two separate systems. The middle and right panels report the passenger and freight flow assignments, showing the passenger and freight loads carried by vehicles across different segments and times.

These figures support the conclusions in Table~\ref{tab:system_comparison}. FCTS is the least flexible system, as fixed compositions and separated operations lead to a rigid timetable and two independent flow structures. MIS and MSS show similar passenger service patterns, consistent with passenger priority in both systems. In both cases, freight is mainly assigned to departures with lower passenger loads to use residual capacity. However, freight operations differ substantially. In MIS, sharing units across passenger and freight functions enables more flexible insertion of freight services with limited additional resources. In MSS, freight units are pre-assigned, which tightens resource availability and constrains freight scheduling. As a result, MSS requires more units and still delivers less efficient freight transport.


Overall, the results demonstrate that capacity sharing and en-route reconfiguration are key operational advantages of the MIS framework, enabling improved efficiency and service quality in integrated passenger–freight systems.

\section{Conclusions and Future Research}
\label{Conclusions and Future Research}

This paper investigates the joint planning and scheduling of modular vehicles for integrated passenger--freight transportation on a bidirectional corridor.
The problem simultaneously determines flexible service routes, operating timetables, and en-route vehicle compositions to serve stochastic passenger demand and deterministic freight demand with different operational requirements.
In addition, we explicitly model unit scheduling to capture unit reuse over the planning horizon and to determine the fleet size required to support the planned services.

To represent these interactions, we construct a space--time--state network and model multiple flow types within a unified network flow framework.
Building on this structure, we develop a stochastic mixed integer programming model.
The passenger assignment component is further reformulated as a path-based model with chance constraints, yielding a linear formulation.

To solve the problem, we propose two solution approaches.
First, we develop an exact Benders decomposition algorithm, in which vehicle planning and demand assignment are handled in the master problem, while the subproblem decomposes into independent unit scheduling problems at individual stations.
To improve computational performance, we incorporate valid inequalities and a warm-start strategy.
Second, to address larger instances, we propose a two-stage decomposition heuristic.
In the first stage, passenger demand is prioritized to construct a feasible passenger service plan.
In the second stage, freight demand and unit scheduling are optimized over the planning horizon conditional on the passenger solution.

We generate two corridor settings inspired by Gothenburg (hub-based and urban) and evaluate the model across various planning horizons and demand scales.
The results show that the exact Benders approach performs well for small- and medium-sized instances, while the heuristic is more efficient for large-scale problems.

Sensitivity analyses examine the effects of passenger-demand uncertainty and the temporal overlap between passenger and freight demands.
The results indicate that accounting for uncertainty is necessary to maintain passenger service quality, whereas more flexible freight scheduling can reduce temporal overlap and improve overall system performance.
Moreover, comparisons with the two benchmark systems show that the proposed system improves operational efficiency by enabling capacity sharing between passenger and freight services and allowing dynamic reconfiguration of vehicle composition along the route.


To focus on the planning and scheduling value of modular vehicle reconfiguration in a pre-planned setting, the model abstracts from detailed operational variability. 
Travel times, docking and undocking times, unit capacities, and station handling capacities are treated as deterministic parameters, while passenger-demand uncertainty is represented through marginal demand-group chance constraints. 
Future research could incorporate richer operational uncertainty and extend the framework to more complex network settings with rolling-horizon or real-time control strategies.

\clearpage
\appendix
\section{Notation Summary}
\label{apx:notation}

This appendix summarizes the key notation used in the mathematical model.

\begin{longtable}{>{\raggedright\arraybackslash}p{0.18\textwidth} >{\raggedright\arraybackslash}p{0.75\textwidth}}
\caption{Sets and indices used in the model.}
\label{tab:notation_sets}\\
\toprule
\textbf{Notation} & \textbf{Description} \\
\midrule
\endfirsthead
\caption[]{Sets and indices used in the model (continued).}\\
\toprule
\textbf{Notation} & \textbf{Description} \\
\midrule
\endhead
\midrule
\multicolumn{2}{r}{Continued on next page}\\
\endfoot
\bottomrule
\endlastfoot

$S$ & Set of stations in the corridor, indexed by $s$ and $s'$. \\
$P$ & Set of passenger demand groups, indexed by $p$. \\
$F$ & Set of freight demand groups, indexed by $f$. \\
$T$ & Set of discrete time points, indexed by $t$ and $t'$. \\
$K$ & Set of operational states, where $K=\{k_1,k_2,k_3\}$ represents upbound, storage, and downbound states. \\
$U$ & Set of unit indices, where $U=\{u_0\}\cup U_f$ includes the virtual passenger unit $u_0$ and freight-unit indices in $U_f$. \\
$U_f$ & Set of freight-unit indices, indexed by $u$. \\
$R_p$ & Set of feasible passenger paths for passenger group $p$. \\
$\mathcal{N}$ & Set of nodes in the space--time--state network. \\
$\mathcal{N}_1,\mathcal{N}_2,\mathcal{N}_3$ & Node subsets for upbound, storage, and downbound states. \\
$\mathcal{E}$ & Set of arcs in the space--time--state network, indexed by $e$. \\
$\mathcal{E}_*,\mathcal{E}_\circ,\mathcal{E}_\triangle$ & Sets of service, storage, and reconfiguration arcs, respectively. \\
$\mathcal{E}_\circ(s)$ & Set of storage arcs at station $s$. \\
$\mathcal{E}_u$ & Set of service and reconfiguration arcs associated with unit index $u$. \\
$\mathcal{E}^{+}(\cdot),\mathcal{E}^{-}(\cdot)$ & Sets of arcs leaving and entering a node, respectively. \\
$\mathcal{E}'^{+}(\cdot),\mathcal{E}'^{-}(\cdot)$ & Outgoing and incoming arcs of a node restricted to $\mathcal{E}'\subseteq\mathcal{E}$. \\
$\mathcal{E}_p,\mathcal{E}_f$ & Sets of admissible arcs for passenger group $p$ and freight group $f$. \\

\end{longtable}

\begin{longtable}{>{\raggedright\arraybackslash}p{0.18\textwidth} >{\raggedright\arraybackslash}p{0.75\textwidth}}
\caption{Parameters used in the model.}
\label{tab:notation_parameters}\\
\toprule
\textbf{Notation} & \textbf{Description} \\
\midrule
\endfirsthead
\caption[]{Parameters used in the model (continued).}\\
\toprule
\textbf{Notation} & \textbf{Description} \\
\midrule
\endhead
\midrule
\multicolumn{2}{r}{Continued on next page}\\
\endfoot
\bottomrule
\endlastfoot

$\underline{t},\overline{t}$ & Start and end times of the planning horizon. \\
\(H\) & Duration of the planning horizon, \(H=\overline{t}-\underline{t}\). \\
$\overline{n}_s$ & Storage capacity of station $s$. \\
$l$ & Maximum number of units in a modular vehicle. \\
$o_p,\delta_p$ & Origin and destination stations of passenger group $p$. \\
$o'_p,\delta'_p$ & Auxiliary source and sink nodes for passenger group $p$. \\
$[a_p,b_p]$ & Arrival time window of passenger group $p$. \\
$k_p$ & Travel direction of passenger group $p$. \\
$\tilde{q}_p$ & Random passenger demand of passenger group $p$. \\
$\mu_p,\sigma_p$ & Mean and standard deviation of the passenger demand distribution of group $p$. \\
$\alpha_p$ & Chance-constraint violation probability for passenger group $p$. \\
$\Phi^{-1}(\cdot)$ & Inverse cumulative distribution function of the standard normal distribution. \\
$o_f,\delta_f$ & Origin and destination stations of freight group $f$. \\
$[a_f,b_f]$ & Delivery time window of freight group $f$. \\
$k_f$ & Travel direction of freight group $f$. \\
$q_f$ & Demand quantity of freight group $f$. \\
$\tau_e$ & Travel, waiting, or processing time of arc $e$. \\
$\tau_r$ & Waiting time associated with passenger path $r$. \\
$c_p,c_f$ & Capacity of one passenger unit and one freight unit, respectively. \\
$\underline{h}$ & Minimum headway between consecutive departures at the same station and direction. \\
$w_1,w_2,w_3$ & Cost coefficients for unit deployment, passenger waiting time, and unmet freight demand, respectively. \\

\end{longtable}

\begin{longtable}{>{\raggedright\arraybackslash}p{0.18\textwidth} >{\raggedright\arraybackslash}p{0.75\textwidth}}
\caption{Decision variables and flow variables used in the model.}
\label{tab:notation_variables}\\
\toprule
\textbf{Notation} & \textbf{Description} \\
\midrule
\endfirsthead
\caption[]{Decision variables and flow variables used in the model (continued).}\\
\toprule
\textbf{Notation} & \textbf{Description} \\
\midrule
\endhead
\midrule
\multicolumn{2}{r}{Continued on next page}\\
\endfoot
\bottomrule
\endlastfoot

$x_{(s,t,k)}$ & Binary variable equal to 1 if a vehicle departs from station $s$ at time $t$ in direction $k$, and 0 otherwise.  \\
$n_s$ & Number of modular units initially deployed at station $s$ at the start of the planning horizon $\underline{t}$. \\
$y_e$ & Number of units assigned to arc $e$. \\
$z_e^p,z_e^f$ & Passenger flow of group $p$ and freight flow of group $f$ assigned to arc $e$, respectively. \\
$z_r^p$ & Passenger flow of group $p$ assigned to feasible path $r$. \\
$\widehat{q}_p$ & Deterministic passenger demand of group $p$ used under the chance constraint. \\

\end{longtable}

\section{Passenger Assignment Model}
\label{apxA}

For each demand scenario, the passenger demand of each group $p \in P$ is randomly generated according to the assumed distribution and treated as given, denoted by $q_p$.
The passenger service plan obtained from the deterministic or chance-constrained models is denoted by $\hat{y}_e$, $e \in \mathcal{E}_{u_0}$, and serves as input to the following assignment model.

Because the capacity provided by a given service plan may be insufficient to cover realized passenger demand in some scenarios, the objective function jointly accounts for passenger waiting costs and penalties for unmet demand.
A large penalty $w_4$ is imposed for each unserved passenger.
\begin{equation}
    \min \quad 
    w_2 \sum_{p \in P} \sum_{r \in R_p} \tau_r z_r^p
    + w_4 \sum_{p \in P} \left(q_p - \sum_{r \in R_p} z_r^p \right).
\end{equation}

The model is subject to the following constraints:
\begin{align}
    &\sum_{r \in R_p} z_r^p \le q_p, && \forall p \in P, \\
    &\sum_{p \in P} \sum_{r \in R_p : e \in r} z_r^p \le c_p \hat{y}_e, && \forall e \in \mathcal{E}_{u_0},\\
    &z_r^p \ge 0, && \forall p \in P, r \in R_p.  \nonumber
\end{align}

The slack demand satisfaction constraints ensure that the assigned flow does not exceed the realized demand. 
In the experiments, the penalty parameter is set to a sufficiently large value ($w_4 = 10{,}000$) to strongly discourage unmet passenger demand.

\section{Modular Segregated System Model}
\label{apxB}

This appendix presents the formulation of the Modular Segregated System (MSS), which serves as a benchmark in the numerical experiments.
Compared with the Modular Integrated System (MIS), MSS retains the modular vehicle structure but prohibits capacity sharing between passenger and freight services.
Passenger and freight units are therefore strictly separated throughout the planning horizon.

In the space--time--state network, all original storage arcs are treated as passenger-unit arcs associated with the passenger-unit type $u_0$.
For each storage arc, we introduce parallel arcs for all freight unit types $u \in U_f$.
As a result, passenger flows are restricted to passenger-dedicated arcs, whereas freight flows can only use freight-dedicated arcs.
To distinguish the initial deployment of the two unit types, we introduce decision variables $n_s^p$ and $n_s^f$, representing the numbers of passenger and freight units initially assigned to station $s$, respectively.

The objective function minimizes the total operating cost:
\begin{equation}
    \min \quad 
    w_1 \sum_{s \in S} (n_s^p+n_s^f)
    + w_2 \sum_{p \in P} \sum_{r \in R_p} \tau_r z_r^p
    + w_3 \sum_{f \in F} \left(q_f - \sum_{e \in \mathcal{E}^{+}(o_f,a_f,k_2)} z_e^f \right).
\end{equation}

The model is subject to the following constraints, which define unit allocation, flow conservation, and capacity limits:
\begin{align}
    &\sum_{u \in U_f}\sum_{e \in \mathcal{E}_u^{+}(s,\underline{t},k_2)} y_e = n_s^f, && \forall s \in S,\\
    &\sum_{e \in \mathcal{E}_{u_0}^{+}(s,\underline{t},k_2)} y_e = n_s^p, && \forall s \in S\\
    &\sum_{u \in U_f}\sum_{e \in \mathcal{E}_u^{-}(s,\overline{t},k_2)} y_e = n_s^f, && \forall s \in S, \\
    &\sum_{e \in \mathcal{E}_{u_0}^{-}(s,\overline{t},k_2)} y_e = n_s^p, && \forall s \in S\\
    &n_s^p + n_s^f \le \overline{n}_s, && \forall s \in S,\\
    &\sum_{e \in \mathcal{E}_{u_0}^{+}(s,t,k)} y_e = \sum_{e \in \mathcal{E}_{u_0}^{-}(s,t,k)} y_e, && \forall (s,t,k) \in \mathcal{N}_2, \; t \neq \underline{t}, \; t \neq \overline{t},  \\
    &\sum_{u \in U_f } \sum_{e \in \mathcal{E}_{u}^{+}(s,t,k)} y_e = \sum_{u \in U_f }  \sum_{e \in \mathcal{E}_{u}^{-}(s,t,k)} y_e, && \forall (s,t,k) \in \mathcal{N}_2, \; t \neq \underline{t}, \; t \neq \overline{t},  \\
    & \sum_{e \in \mathcal{E}_\circ^+(s,t,k)} y_e \le \overline{n}_s, && \forall (s,t,k) \in \mathcal{N}_2, \\ 
    &n_s^p,n_s^f \in \mathbb{Z}_{\geq 0}, && \forall s \in S, \\
    & (\ref{c1}), (\ref{c6})-(\ref{c8}), (\ref{c10}), (\ref{c12}), (\ref{c18})-(\ref{c23}), (\ref{c26})-(\ref{c28}), (\ref{c29})-(\ref{c30}). \nonumber
\end{align}

These additional constraints enforce the strict separation between passenger and freight units.
The initial and terminal allocation constraints distinguish the two unit types and ensure feasibility with respect to station capacity.
Flow balance is imposed separately for passenger and freight units at storage nodes, eliminating any possibility of capacity sharing.
The remaining constraints are identical to those in MIS and govern vehicle operations, demand satisfaction, and integrality requirements.

\section{Fixed-Composition Transit System Model}
\label{apxC}

This appendix presents the formulation of the Fixed-Composition Transit System (FCTS), which serves as a benchmark in the numerical experiments.
Starting from the space--time--state network, the conventional fixed-capacity transit setting is obtained by modifying the original model.

In FCTS, passenger and freight services are operated as two separate systems and solved independently.
For the passenger system, all freight-related arcs are disabled.
In addition, all reconfiguration arcs at non-terminal nodes are disabled, and the pool size at each non-terminal station is set to zero.
Consequently, no unit storage or reconfiguration is allowed at intermediate stations.
Vehicles can only depart from and return to terminal stations, and their composition remains fixed along the route.
Docking and undocking times at terminals are set to zero because each vehicle is treated as an indivisible unit and no reconfiguration is modeled.

The objective function for the passenger system is
\begin{equation}
    \min \quad
    w_1 \sum_{s \in S} n_s
    + w_2 \sum_{p \in P} \sum_{r \in R_p} \tau_r z_r^p.
\end{equation}

The model is subject to the following additional constraints:
\begin{align}
    &y_e = 0, && \forall e \in \mathcal{E}_u,\; u \in U_f, \\
    &\sum_{e \in \mathcal{E}_\triangle^+(s,t,k)} y_e + \sum_{e \in \mathcal{E}_\triangle^-(s,t,k)} y_e  = 0, && \forall (s,t,k) \in \mathcal{N}_2, \ s \in S^{\mathrm{int}},\label{apx:con1}\\
    &\sum_{e \in \mathcal{E}_*^{+}(s,t,k)} y_e = l\, x_{(s,t,k)}, && \forall (s,t,k) \in \mathcal{N}_1 \cup \mathcal{N}_3, \label{apx:con2}\\
    &(\ref{c1})-(\ref{c6}), (\ref{c9})-(\ref{c12}), (\ref{c26})-(\ref{c28}), (\ref{c29})-(\ref{c30}). \nonumber
\end{align}

The first constraint removes all freight-unit arcs.
The second constraint prohibits docking and undocking at intermediate stations, denoted by \(S^{\mathrm{int}}\). 
Since \(y_e \ge 0\), forcing the total incoming and outgoing reconfiguration flow at each intermediate storage node to zero disables all associated reconfiguration arcs.
The third constraint fixes the number of units on each service arc, ensuring a constant vehicle composition.
In the experiments, passenger vehicles operate with fixed five-unit consists (i.e., $l=5$).
The remaining constraints are inherited from the original model.

The freight system is formulated analogously and solved independently.
All passenger-related arcs are disabled, and a simplified flow structure is adopted for freight transportation.
Specifically, freight flows are not distinguished by individual units. Instead, a single aggregated arc is constructed for each service and reconfiguration activity, with nonnegative flow representing the total freight capacity.

The objective function for the freight system is
\begin{equation}
    \min \quad
    w_1 \sum_{s \in S} n_s
    + w_3 \sum_{f \in F} \left(q_f - \sum_{e \in \mathcal{E}^{+}(o_f,a_f,k_2)} z_e^f \right).
\end{equation}

The corresponding constraints are:
\begin{align}
    &y_e = 0, && \forall e \in \mathcal{E}_{u_0}, \\
    &(\ref{c1})-(\ref{c6}), (\ref{c9})-(\ref{c12}), (\ref{c18})-(\ref{c23}), (\ref{apx:con1})-(\ref{apx:con2}). \nonumber
\end{align}

In this case, freight vehicles operate with fixed three-unit consists (i.e., $l=3$).

The passenger and freight systems are solved independently, and the results are combined to obtain the total unit requirement and the corresponding service performance.

\clearpage
\bibliography{references.bib} 

\end{document}